\documentclass[10pt]{article}
\usepackage{mathrsfs}
\usepackage{amssymb}
\usepackage{amsmath}
\usepackage[all]{xy}
\usepackage{amsthm}
\usepackage{amsfonts}
\usepackage{latexsym}
\usepackage{amscd,amssymb,amsopn,amsmath,amsthm,graphics,amsfonts,mathrsfs,accents,enumerate,verbatim,calc}
\usepackage[dvips]{graphicx}
\usepackage{cite}
\usepackage[colorlinks=true,linkcolor=red,citecolor=blue]{hyperref}
\newtheorem{thm}{Theorem}[section]
\newtheorem{prop}[thm]{Proposition}

\newtheorem{lem}[thm]{Lemma}

\newtheorem{cor}[thm]{Corollary}
\newtheorem{df}[thm]{Definition}
\newtheorem{exa}[thm]{Example}

\newtheorem{rem}[thm]{Remark}

\newtheorem{Para}[thm]{}

\newcommand{\X}{\mathscr{X}}

\newcommand{\T}{\mathcal{T}}
\newcommand{\N}{\mathcal{N}}

\newcommand{\K}{\mathbb{K}}
\newcommand{\D}{\mathbb{D}}

\def\Im{\mathop{\rm Im}\nolimits}
\def\Ker{\mathop{\rm Ker}\nolimits}
\def\Coker{\mathop{\rm Coker}\nolimits}

\def\mod{\mathop{\rm mod}\nolimits}

\def\Hom{\mathop{\rm Hom}\nolimits}

\def\Ext{\mathop{\rm Ext}\nolimits}

\def\pd{\mathop{\rm pd}\nolimits}

\def\Gpd{\mathop{\rm Gpd}\nolimits}

\def\proj{\mathop{\rm proj}\nolimits}

\def\Gproj{\mathop{\rm Gproj}\nolimits}
\def\Gperf{\mathop{\rm Gperf}\nolimits}

\def\Gpd{\mathop{\rm Gpd}\nolimits}
\def\inf{\mathop{\rm inf}\nolimits}

\def\Im{\mathop{\rm Im}\nolimits}
\begin{document}

\title{\large \bf 2-recollements of singualrity categories and Gorenstein defect categories over triangular matrix algebras
\thanks{{\it 2010 Mathematics Subject Classification}: 18E30, 18E35, 18G20.}
\thanks{{\it Keywords}: recollement; 2-recollement; singularity category; Gorenstein defect category; triangular matrix algebra.}
}
\author{Huanhuan Li$^a$, Dandan Yang$^a$, Yuefei Zheng$^b$ and Jiangsheng Hu$^{c}$\footnote{Corresponding author} \\
\it\footnotesize $^a$School of Mathematics and Statistics, Xidian University, Xi'an 710071, Shaanxi Province, China\\
\it\footnotesize $^b$College of Science, Northwest A$\&$F University, Yangling 712100, Shaanxi Province, China\\
\it\footnotesize $^c$School of Mathematics and Physics, Jiangsu University of Technology, Changzhou 213001, Jiangsu Province, China\\
%\it\small $^*$Corresponding author.\\
\it\footnotesize Email addresses: lihh@xidian.edu.cn, ddyang@xidian.edu.cn, yuefeizheng@sina.com and jiangshenghu@jsut.edu.cn
}
\date{}
\baselineskip=14pt
\maketitle
\begin{abstract} Let $T=\left(
                                                                                 \begin{array}{cc}
                                                                                   A & M \\
                                                                                   0 & B \\
                                                                                 \end{array}
                                                                               \right)$ be a triangular matrix algebra with its corner algebras $A$ and $B$ Artinian and $_AM_B$ an $A$-$B$-bimodule. The 2-recollement structures for singularity categories and Gorenstein defect categories over $T$ are studied. Under mild assumptions, we provide necessary and
                                                                               sufficient conditions for the existences of 2-recollements of singularity categories and Gorenstein defect categories
                                                                               over $T$ relative to those of $A$ and $B$. Parts of our results strengthen and unify the corresponding work in \cite{LL,L,ZP}.
\end{abstract}

%\centerline { \bf  Abstract}
%%\bigskip
%\noindent For an abelian category $\mathscr{A}$ with enough projective objects, we prove that
%
%We study the Gorenstein projective dimensions of complexes in an abelian category \\
%\vbox to 0.3cm{}\\
%{\it Key Words:} Gorenstein projective dimensions; triangulated categories; singularity category; triangle-equivalence; recollements.\\
%{\it 2010 Mathematics Subject Classification:} 16E05; 18G20; 18G35.

\section{\bf Introduction}
The singularity category was introduced by Buchweitz, known back then as the stable derived category, in his famous unpublished paper \cite{Bu}. As an initial purpose, Buchweitz used this category to study the stable homological algebra and Tate cohomology for certain rings. In the setting of algebraic geometry, this category was reconsidered by Orlov \cite{O} and turned out to have a closed relation with the ``Homological Mirror Symmetry Conjecture'' due to Kontsevich. Recall that, for a given algebra $R$, the singularity category $\D_{sg}(R)$ of $R$ is defined to be the Verdier quotient $\D_{sg}(R):=\D^b(\mod R)/\K^b(\proj R)$, where $\D^b(\mod R)$ is
the bounded derived category of finitely generated $R$-modules and $\K^b(\proj R)$ is the bounded homotopy category of finitely generated projective $R$-modules (i.e., the subcategory of perfect complexes). It measures the ``regularity'' of $R$ in sense that $\D_{sg}(R)=0$ if and only if $R$ is of finite global dimension.
By the fundamental result in \cite{Bu}, the singularity category contains $\underline{\Gproj R}$ (the stable category of finitely generated Gorenstein projective $R$-modules) as a triangulated subcategory.
This means there exists a fully faithful triangle functor $F:\underline{\Gproj R}\to\D_{sg}(R)$; besides, $F$ is a triangle-equivalence provided that $R$ is Gorenstein \cite{Bu,H2}. Motivated by this, Bergh, J{\o}rgensen and Oppermann \cite{BJO} introduced the Verdier quotient $\D_{def}(R):=\D_{sg}(R)/\Im F$, and they called it the Gorenstein defect category of $R$.
This category measures how far the algebra $R$ is from being Gorenstein. More precisely, $R$ is Gorenstein if and only if $\D_{def}(R)$ is trivial. Recently, singularity categories and related topics have been studied by many authors, see for example \cite{C1,KZ,LZHZ,LHZ,LL,L,PSS,ZP,ZH}.  % ,\cite{C1},\cite{KZ}-\cite{L},\cite{PSS},\cite{ZP},\cite{ZH}. %\cite{LZHZ},\cite{LHZ},\cite{LL},\cite{L}

Recollements of triangulated categories and abelian categories arise constantly in algebraic geometry and representation theory \cite{AKLY,BBD,CX1,CX2,CL,CPS2,Hap,M2,P,PSS}. Roughly speaking, a recollement is a short exact sequence of triangulated or abelian categories where the functors involving admit both left and right adjoints. Recollements were first introduced in the setting of triangulated categories by Beilinson, Bernstein
and Deligne \cite{BBD} and then generalized to the level of abelian categories (see e.g. \cite{Hap,P,PSS}). This technique provides a categorical reduction for a bigger triangulated or abelian category to decompose into two smaller ones. Consequently, one might obtain certain algebraic properties of the middle term from the outer two smaller ones. There has been lots of people who consider when the recollement admits some extra adjoint functors.
For example, if there exists a (right) Serre functor in the middle term, then the recollement can be extended $1$ step downwards and  $1$ step upwards \cite{J}; in this case, the diagram involving is called a symmetric recollement \cite{ZP}. If the recollement can be extended $n$ steps downwards, then the diagram involving is called an $n$-recollement \cite{QH}. While the diagram involving is called a ladder \cite{AKLY,BGS}, if the recollement could be extended upwards and downwards. Generous evidences indicate that a recollement behaves better when it admits some extra adjoint functors, see \cite{AKLY,BGS,QH,ZZZZ} and references therein for instance.

Let $T=\left(
                                                                                 \begin{array}{cc}
                                                                                   A & M \\
                                                                                   0 & B \\
                                                                                 \end{array}
                                                                               \right)$ be a triangular matrix algebra with its corner algebras $A$ and $B$ Artinian and $_AM_B$ an $A$-$B$-bimodule. The study of singularity theory
over $T$ by recollements has been considered by many people. For instance, Zhang characterized in \cite{ZP} the class of Gorenstein projective $T$-modules. As an application, he showed that if $T$ is a Gorenstein algebra and $_AM$ is projective, then there exists a recollement of $\underline{\Gproj T}$ relative to $\underline{\Gproj A}$ and $\underline{\Gproj B}$. Later on, Liu-Lu \cite{LL} and Lu \cite{L} generalized this to consider the singularity categories and Gorenstein defect categories, respectively. More precisely, they provided sufficient conditions for the existence of a recollement $\D_{sg}(T)$ (resp. $\D_{def}(T)$) relative to $\D_{sg}(A)$ (resp. $\D_{def}(A)$) and $\D_{sg}(B)$ (resp. $\D_{def}(B)$). However, the results mentioned about provided only sufficient conditions for the existences of certain recollements. So we wonder whether or not we can get necessary and sufficient conditions for the existences of such recollements. Besides, the recollement structures over the triangular matrix algebra might be enriched in some suitable settings. For example,
it was shown in \cite{ZZZZ} that if $A$, $B$ and $T$ are finite-dimensional Gorenstein algebras, then there exists a unbounded ladder of period 1 for the stable categories of Gorenstein projective modules (and hence for the singularity categories). Meanwhile, we note that the recollements under consideration in \cite{ZP,LL,L} are initially from the following 2-recollement of module categories:
$$\xymatrix@=3cm{\mod A\ar@/_3pc/[r]_{\Hom_A(e_AT,-)}\ar@/^1pc/[r]^{i_{e_A}}& \mod T\ar@/_3pc/[r]_{\Hom_T(B,-)}\ar@/^1pc/[l]^{S_{e_A}}\ar@/_3pc/[l]_{A\otimes_T-}\ar@/^1pc/[r]^{S_{e_B}} & \mod B\ar@/^1pc/[l]^{i_{e_B}}\ar@/_3pc/[l]_{Te_B\otimes_B-},}\eqno{(3.1)}$$
see Lemma \ref{lem:3.1} for the detailed expressions of these functors. Therefore, the 2-recollements for the singularity categories, Gorenstein defect categories and stable categories of Gorenstein projective modules over $T$ are expected. In this present paper, we aim to solve these questions.  More precisely, we get the following main results.
%Our main purpose of this paper is to study these questions

\begin{thm}\label{thm:1.1} Let $T=\left(
                                                                                 \begin{array}{cc}
                                                                                   A & M \\
                                                                                   0 & B \\
                                                                                 \end{array}
                                                                               \right)$ be a triangular matrix algebra with $_AM_B$ an $A$-$B$-bimodule. Assume that $\pd_A M<\infty$, $\pd M_B<\infty$ and $M\in{^\bot A}$. Then we have the following 2-recollement of singularity categories:
$$\xymatrix@=3cm{\D_{sg}(A)\ar@/_3pc/[r]_{\overline{\mathbb{R}\Hom_A(e_AT,-)}}\ar@/^1pc/[r]^{\overline{\D^b(i_{e_A})}}& \D_{sg}( T)\ar@/_3pc/[r]_{\overline{\mathbb{R}\Hom_T(B,-)}}\ar@/^1pc/[l]^{\overline{\D^b(S_{e_A})}}\ar@/_3pc/[l]_{\overline{A\otimes^\mathbb{L}_T-}}\ar@/^1pc/[r]^{\overline{\D^b(S_{e_B})}} & \D_{sg}( B)\ar@/^1pc/[l]^{\overline{\D^b(i_{e_B})}}\ar@/_3pc/[l]_{\overline{Te_B\otimes^\mathbb{L}_B-}}}\eqno{(1.1)}$$
if and only if $\pd_B\Hom_A(M,A)<\infty$, where all these functors are initially from (3.1) (see Propositions \ref{prop:3.3} and \ref{prop:3.5} for the detailed descriptions).
\end{thm}

Recall from \cite{ZP} that $_AM_B$ is {\it compatible} if $M\otimes_B-$ carries every acyclic complex of projective $B$-modules to acyclic $A$-complex and $M\in(\Gproj A)^\perp$. We call an $A$-$B$-bimodule $_AM_B$ {\it left Gorenstein singular} if $\Gpd_B\Hom_A(M,F)<\infty$ for any $F\in\Gproj A$; while we call $_AM_B$ {\it right Gorenstein singular} if $\Gpd_AM\otimes_BG<\infty$ for any $G\in\Gproj B$. $_AM_B$ is said to be {\it Gorenstein singular} if it is both left and right Gorenstein singular. We have the following equivalent characterizations for the existence of a 2-recollement of Gorenstein defect categories.

\begin{thm}\label{thm:1.2} Let $T=\left(
                                                                                 \begin{array}{cc}
                                                                                   A & M \\
                                                                                   0 & B \\
                                                                                 \end{array}
                                                                               \right)$ be a triangular matrix algebra with $_AM_B$ compatible. Assume that $\pd_A M<\infty$, $M\in{^\bot\Gproj A}$ and $\pd M_B<\infty$. Then we have the following 2-recollement of Gorenstein defect categories:
$$\xymatrix@=3cm{\D_{def}(A)\ar@/_3pc/[r]_{\widetilde{\mathbb{R}\Hom_A(e_AT,-)}}\ar@/^1pc/[r]^{\widetilde{\D^b(i_{e_A})}}& \D_{def}( T)\ar@/_3pc/[r]_{\widetilde{\mathbb{R}\Hom_T(B,-)}}\ar@/^1pc/[l]^{\widetilde{\D^b(S_{e_A})}}\ar@/_3pc/[l]_{\widetilde{A\otimes^\mathbb{L}_T-}}\ar@/^1pc/[r]^{\widetilde{\D^b(S_{e_B})}} & \D_{def}( B)\ar@/^1pc/[l]^{\widetilde{\D^b(i_{e_B})}}\ar@/_3pc/[l]_{\widetilde{Te_B\otimes^\mathbb{L}_B-}}}\eqno{(1.2)}$$
if and only if $M$ is Gorenstein singular, where all these functors are initially from (3.1) (see Propositions \ref{prop:4.4} and \ref{prop:4.7} for the detailed descriptions).
\end{thm}

In the procedure of proving Theorems \ref{thm:1.1} and \ref{thm:1.2}, we obtain equivalent characterizations for the existences of recollements of singularity categories and Gorenstein defect categories over $T$, which generalize
the corresponding results in \cite{LL,L} (see Propositions \ref{prop:3.3} and \ref{prop:4.4}). As a consequence, the recollements (see Corollaries \ref{cor:4.5} and \ref{cor:4.8}) and 2-recollements (see Corollary \ref{cor:4.9}) of stable categories of Gorenstein projective modules over $T$ are obtained accordingly, where Corollary \ref{cor:4.5} generalizes Zhang's result to a more general case (compare \cite[Theorem 3.5]{ZP}).

The contents of this paper are outlined as follows.
In Section \ref{Pre}, we fix some notations and recall some basic definitions and facts that are needed in the later proofs.
In Section \ref{Bir}, we consider the 2-recollements of singularity categories over the triangular matrix algebra and prove Theorem \ref{thm:1.1}.
In Section \ref{Gd}, the 2-recollements for Gorenstein defect categories and stable categories of Gorenstein projective modules over the triangular matrix algebra are studied, including the proof of Theorem \ref{thm:1.2}.

\section{Preliminaries}\label{Pre}

In this section, we briefly recall some basic definitions, facts and notations needed in the sequel.

\vspace{0.2cm}
{\bf 2.1. Notations and conventions}
\vspace{0.2cm}

Throughout, all algebras are Artin algebras over a fixed commutative Artinian ring and all modules are finitely generated. For a given algebra $R$,
denote by $\mod R$ the category of left $R$-modules; right $R$-modules are viewed as left $R^{op}$-modules, where $R^{op}$ is the opposite algebra of $R$.
We use $\proj R$ to denote the subcategory of $\mod R$ consisting of projective modules. The $*$-bounded derived category of $\mod R$ and homotopy category of $\proj R$
will be denoted by $\D^*(\mod R)$ and $\K^*(\proj R)$ respectively, where $*\in\{blank,\ +,\ -, \ b\}$.

Usually, we use $_RM$ (resp. $M_R$) to denote a left (resp. right) $R$-module $M$, and the projective dimension of $_RM$ (resp. $M_R$) will be denoted by $\pd_RM$ (resp. $\pd M_R$).
For a subclass $\X$ of $\mod R$. Denote by $\X^\bot$ (resp. $^\bot\X$) the subcategory consisting of modules $M\in \mod R$ such that $\Ext_R^n(X,M)=0$ (resp. $\Ext_R^n(M,X)=0$) for any $X\in\X$ and $n\geq1$.
%
%\vspace{0.2cm}
%Recall from \cite{BBD} that a {\it recollement of triangulated categories} is a diagram of triangulated categories and triangle-functors
%$$\xymatrix{\T'\ar[r]^{i_*}&\T\ar@/^1pc/[l]^{i^!}\ar@/_1pc/[l]_{i^*}\ar[r]^{j^*} &\T''\ar@/^1pc/[l]^{j_*}\ar@/_1pc/[l]_{j_!} }\eqno{(2.1)}$$
%satisfying:
%
%(R1) $(i^*,i_*)$, $(i_*,i^!)$, $(j_!,j^*)$ and $(j^*,j_*)$ are adjoint pairs;
%
%(R2) $i_*$, $j_!$ and $j_*$ are fully faithful;
%
%(R3) $j^*i_*=0$;
%
%(R4) For each $X\in\T$, there are two triangles in $\T$ induced by counit and unit adjunctions:
%
%$$i_*i^!X\to X\to j_*j^*X\to i_*i^!X[1] \ \textrm{and} \ j_!j^*X\to X\to i_*i^*X\to j_!j^*X[1],$$
%
% where $[1]$ is the shift functor of $\T$.
%
%\vspace{2mm}
%By a {\it left recollement of triangulated categories}, we mean a diagram of triangulated categories and triangle-functors consisting of the upper two rows of (2.1), satisfying conditions (R1)-(R4) which involve only the functors
%$i^*,i_*,j_!$ and $j^*$.
%
%We call diagram (2.1) a {\it recollement of abelian categories}, if $\T, \T', \T''$ in diagram (2.1) are abelian categories, the six functors involving are additive functors and conditions (R1), (R2) and (R5) are satisfied, where
%
%(R5) $\Im i_*=\Ker j^*.$
%
%\begin{lem}(\cite{P,PSS})\label{lem:2.1}  Let $R$ be an Artin algebra and $e\in R$ an idempotent. Then we have the following recollement of module categories:
%$$\xymatrix@=1cm{\mod R/ReR\ar[r]^{i_*}& \mod R\ar@/^1pc/[l]^{i^!}\ar@/_1pc/[l]_{i^*}\ar[r]^{j^*} & \mod eRe\ar@/^1pc/[l]^{j_*}\ar@/_1pc/[l]_{j_!}.}$$
%\end{lem}

\vspace{0.2cm}

Let $$X^\bullet=\cdots\to X^{-1}\xrightarrow{d^{-1}}X^{0}\xrightarrow{d^{0}} X^{1}\to\cdots$$
be a complex in $\mod R$. For any integer $n$, we set $Z^n(X^\bullet)=\Ker d^n$, $B^n(X^\bullet)=\Im d^{n-1}$ and $H^n(X^\bullet)=Z^n(X^\bullet)/B^n(X^\bullet)$. $X^{\bullet}$ is called  acyclic (or exact) if $H^{n}(X^{\bullet})=0$ for any $n\in \mathbb{Z}$.

\vspace{0.2cm}
{\bf  2.2.  Gorenstein projective modules and Gorenstein perfect complexes}
\vspace{0.2cm}

Recall from \cite{AB,AM,Ho} that an acyclic complex $X^\bullet$ is called {\it totally acyclic} if each $X^i\in\proj R$ and $\Hom_R(X^\bullet,R)$ is acyclic. A module $M\in\mod R$ is {\it Gorenstein projective} if there exists some totally acyclic complex $X^\bullet$ such that $M\cong Z^0(X^\bullet)$.
Denote by $\Gproj R$ the subcategory of $\mod R$ consisting of Gorenstein projective modules.
Given a module $M\in\mod R$, the {\it Gorenstein projective dimension} $\Gpd_RM$ of $M$ is defined to be $\Gpd_RM=\inf\{n:$ there exists an exact sequence $0\to G_n\to \cdots\to G_1\to G_0\to M\to0$, where each $G_i\in\Gproj R$$\}$.

\begin{df}\label{df:2.4} (compare \cite{LZHZ}) {\rm A complex $X^\bullet\in\D^b(\mod R)$ is said to be {\it Gorenstein perfect} if $X^\bullet$ is isomorphic to some bounded complex consisting of Gorenstein projective modules in $\D^b(\mod R)$.
}\end{df}

Denote by $\Gperf(R)$ the subcategory of $\D^b(\mod R)$ consisting of Gorenstein perfect complexes.

\begin{rem}\label{rem:2.3} {\rm A Gorenstein perfect complex is called a complex with finite Gorenstein projective dimension in \cite{LZHZ}. Here we use the name of ``Gorenstein perfect'' because we find this kind of complexes reflects as the Gorenstein version of perfect complexes. For instance, it is not hard to see an $R$-module $M$ is Gorenstein perfect if and only if $\Gpd_RM<\infty$. Besides, $\Gperf(R)$ is the smallest thick subcategory of $\D^b(\mod R)$ containing $\Gproj R$. For more details, we refer the reader to appendix in \cite{LZHZ}.
}\end{rem}

\begin{lem}\label{lem:2.4}(see \cite[Proposition A.4]{LZHZ})  Let $X^\bullet\in\D^b(\mod R)$. If each $X^i$ is of finite Gorenstein projective dimension as an $R$-module, then $X^\bullet\in\Gperf(R)$.
\end{lem}

%\begin{lem}\label{lem:2.1}
%Let $X^\bullet$ be an acyclic complex with each $X^i\in\proj R$.
%Assume that $M\in\mod R^{op}$ with $\pd M_R<\infty$, then $M\otimes_R X^\bullet$ is acyclic. In this case, $\Tor^R_i(M,G)=0$ for any $G\in\Gproj R$ and any integer $i\geq1$.
%\end{lem}
%\begin{proof}
%Let $X^\bullet$ be an acyclic complex with each $X^i\in\proj R$ and $\pd M_R<\infty$. Then there exists an exact sequence $0\to P_n\to \cdots\to P_1\to P_0\to M\to0$ with each $P_i\in\proj R^{op}$ for some integer $n\geq0$. Since each $X^i\in\proj R$, we get an exact sequence of complexes of abelian groups $0\to P_n\otimes_RX^\bullet\to\cdots\to P_1\otimes_RX^\bullet\to P_0\otimes_RX^\bullet\to M\otimes_RX^\bullet\to0$.
%Note that each $P_i\otimes_RX^\bullet$ is acyclic. Then $M\otimes_RX^\bullet$ is acyclic. Now let $G\in\Gproj R$, then there is a totally acyclic complex $P^\bullet$ such that $G\cong Z^0(P^\bullet)$. It follows that $M\otimes_R P^\bullet$ is acyclic and hence $\Tor^R_i(M,G)=0$ for any integer $i\geq1$.
%\end{proof}

\vspace{0.2cm}
{\bf 2.3. Singularity categories and Gorenstein defect categories}

Recall that the singularity category $\D_{sg}(R)$ of $R$ is defined to be the verdier quotient
$$\D_{sg}(R):=\D^b(\mod R)/\K^b(\proj R),$$ where complexes in $\K^b(\proj R)$ (up to isomorphisms) are the so-called {\it perfect complexes}.
This category was first introduced by Buchweitz \cite{Bu}, and later reconsidered by a lot of authors \cite{Be1,H2,O}. It is well-known that
$\Gproj R$ is a Frobenius category, and hence its stable category $\underline{\Gproj R}$ is a triangulated category \cite{H1}. By a fundamental result of Buchweitz, there exists
a fully faithful triangle functor $F:\underline{\Gproj R}\to\D_{sg}(R)$, which sends every Gorenstein projective module to the stalk complex concentrated in degree zero. Furthermore, $F$ is a triangle-equivalence provided that $R$ is Gorenstein (that is, the left and right self-injective dimensions of $R$ are finite). Consequently, $\Im F$ is a triangulated subcategory of $\D_{sg}(R)$. Following \cite{BJO}, the Verdier quotient $$\D_{def}(R):=\D_{sg}(R)/\Im F$$ is called the {\it Gorenstein defect category} of $R$.

\begin{lem}\label{lem:2.5}(see \cite[Theorem A.5]{LZHZ}) We have the following exact commutative diagram:
$$\xymatrix{0\ar[r]&\underline{\Gproj R}\ar[r]^F\ar[d]
&\D_{sg}(R)\ar[r]\ar@{=}[d]
& \D_{def}(R)\ar[d]\ar[r]&0\\
0\ar[r]&{\Gperf(R)/\K^b(\proj R)}\ar[r]
&\D_{sg}(R)\ar[r]
&{\D^b(\mod R)/\Gperf(R)}\ar[r]&0
}$$
with all vertical functors triangle-equivalences.
\end{lem}

\vspace{0.2cm}
{\bf 2.4. Recollements and 2-recollements}
\vspace{0.2cm}

Let $\T, \T'$ and $\T''$ be triangulated categories. A {\it recollement} \cite{BBD} of $\T$ relative to $\T'$ and $\T''$ is a diagram of triangulated categories and triangle functors
$$\xymatrix{\T'\ar[r]^{i_*}&\T\ar@/^1pc/[l]^{i^!}\ar@/_1pc/[l]_{i^*}\ar[r]^{j^*} &\T''\ar@/^1pc/[l]^{j_*}\ar@/_1pc/[l]_{j_!} }\eqno{(2.1)}$$
satisfying:

(R1) $(i^*,i_*)$, $(i_*,i^!)$, $(j_!,j^*)$ and $(j^*,j_*)$ are adjoint pairs;

(R2) $i_*$, $j_!$ and $j_*$ are fully faithful;

(R3) $\Im i_*=\Ker j^*.$\\
%It is not hard to see if (R1) and (R2) are satisfied, then (R3) is equivalent to the following two conditions (R4) and (R5):
%
%(R4) $j^*i_*=0$;
%
%(R5) For each $X\in\T$, there are two triangles in $\T$ induced by counit and unit adjunctions:
%$$i_*i^!X\to X\to j_*j^*X\to i_*i^!X[1] \ \textrm{and} \ j_!j^*X\to X\to i_*i^*X\to j_!j^*X[1],$$
%
% where $[1]$ is the shift functor of $\T$.

If $\T$, $\T'$ and $\T''$ in diagram (2.1) are abelian categories, and the six functors involving are additive functors. Then we call diagram (2.1) a {\it recollement of abelian categories}, see \cite{Hap,P,PSS} for details.

\begin{df}{\rm (see \cite{QH})  Let  $\T, \T'$ and $\T''$ be triangulated categories (resp. abelian categories). A {\it 2-recollement} of $\T$ relative to $\T'$ and $\T''$ is given by a diagram
$$\xymatrix{\T'\ar@/_2pc/[r]_{i_?}\ar@/^0.5pc/[r]^{i*}&\T\ar@/_2pc/[r]_{j^?}\ar@/^0.5pc/[l]^{i^!}\ar@/_2pc/[l]_{i^*}\ar@/^0.5pc/[r]^{j^*} &\T''\ar@/^0.5pc/[l]^{j_*}\ar@/_2pc/[l]_{j_!} }\eqno{(2.2)}$$
such that every consecutive three layers form a recollement.
%In this case, the recollement involving the upper three rows is called the {\it upper-recollement}; while the recollement involving the lower three rows is called the {\it lower-recollement}.
}\end{df}

We remark that a 2-recollement is also called a ladder of height 2 in the sense of \cite{AKLY}, see also \cite{BGS} for instance.

\begin{lem}\label{lem:2.7} (compare \cite{LL}) Let (2.1) be a recollement of triangulated categories. Assume that $\N$, $\N'$ and $\N''$  are thick subcategories of $\T$, $\T'$ and $\T''$ respectively. The following statements are equivalent:
\begin{enumerate}
\item[(1)] (2.1) restricts to the following recollement: $$\xymatrix{\N'\ar[r]^{i_*}&\N\ar@/^1pc/[l]^{i^!}\ar@/_1pc/[l]_{i^*}\ar[r]^{j^*} &\N''\ar@/^1pc/[l]^{j_*}\ar@/_1pc/[l]_{j_!}. }\eqno{(2.3)}$$
\item[(2)] (2.1) induces the following recollement: $$\xymatrix{\T'/\N'\ar[r]^{\overline{{i_*}}}&\T/\N\ar@/^1pc/[l]^{\overline{i^!}}\ar@/_1.2pc/[l]_{\overline{i^*}}\ar[r]^{\overline{j^*}} &\T''/\N''\ar@/^1pc/[l]^{\overline{j_*}}\ar@/_1.2pc/[l]_{\overline{j_!}}, }\eqno{(2.4)}$$ where these six functors are induced by those in (2.1).
\item[(3)] $i^*(\N)\subseteq\N'$, $i_*(\N')\subseteq\N$, $j^*(\N)\subseteq\N''$ and $j_*(\N'')\subseteq\N$.
\end{enumerate}
\end{lem}
\begin{proof} $(1)\Rightarrow(3)$ and $(2)\Rightarrow(3)$ are trivial. Now assume that conditions in (3) are satisfied, we claim $i^*(\N)=\N'$ and $j^*(\N)=\N''$. To do this, let $X'\in\N'$. Put $X=i_*X'$, it follows that $X\in\N$. Hence $X'\cong i^*i_*(X')\cong i^*(X)$ and then $X'\in i^*(\N)$. Thus $i^*(\N)=\N'$ as desired. Similarly, one could obtain $j^*(\N)=\N''$. Therefore, the claim follows. Now we infer $(3)\Rightarrow(1)$ and $(3)\Rightarrow(2)$
from \cite[Remark 2.4 and Proposition 2.5]{LL}.
\end{proof}

Let $R$ be an Artin algebra and $e\in R$ an idempotent. Recall from \cite[Chapter 6]{G} that the Schur functor $S_e:\mod R\to \mod eRe$ associative to $e$ is defined to be $S_e(X)=eX$ for any $X\in\mod R$.
Clearly, $S_e$ admits a fully faithful left adjoint $Re\otimes_{eRe}-:\mod eRe\to\mod R$
and a fully faithful right adjoint
$\Hom_{eRe}(eR,-):\mod eRe\to\mod R.$
Denote by $i_{1-e}:\mod R/ReR\to\mod R$ the canonical inclusion functor induced by the natural homomorphism $R\to R/ReR$. We have the following

\begin{exa}\label{exa:2.8} Let $R$ be an Artin algebra and $e\in R$ an idempotent.
\begin{enumerate}
\item[(1)] (see \cite{P,PSS}) We have the following recollement of module categories:
$$\xymatrix@=2.5cm{\mod R/ReR\ar[r]^{i_{1-e}}& \mod R\ar@/^1pc/[l]^{\Hom_R(R/ReR,-)}\ar@/_1.5pc/[l]_{R/ReR\otimes_R-}\ar[r]^{S_e} & \mod eRe\ar@/^1pc/[l]^{\Hom_{eRe}(eR,-)}\ar@/_1.5pc/[l]_{Re\otimes_{eRe}-}.}\eqno{(2.5)}$$
\item[(2)]  (see \cite{CPS2,M2}) We have the following recollement of bounded derived categories:
$$\xymatrix@=2.5cm{\D^b(\mod R/ReR)\ar[r]^{\D^b(i_{1-e})}&\D^b( \mod R)\ar@/^1pc/[l]^{\mathbb{R}\Hom_R(R/ReR,-)}\ar@/_1.5pc/[l]_{R/ReR\otimes^\mathbb{L}_R-}\ar[r]^{\D^b(S_e)} &\D^b(\mod eRe )\ar@/^1pc/[l]^{\mathbb{R}\Hom_{eRe}(eR,-)}\ar@/_1.5pc/[l]_{Re\otimes^\mathbb{L}_{eRe}-} }\eqno{(2.6)}$$ such that all functors are the derived versions of those in (2.5)
if and only if the following conditions are satisfied:
(i) $\Ext_R^n(R/ReR,R/ReR)=0$ for every integer $n\geq1$;
(ii) $\pd_RR/ReR<\infty$;
(iii) $\pd {R/ReR}_R<\infty$.
%In this case, we have $\pd_{eRe}eR<\infty$ and $\pd Re_{eRe}<\infty$.
\end{enumerate}
\end{exa}

%We have the following observation.
%
%\begin{lem} Let (3.1) be a recollement of triangulated categories. Assume $\mathcal{N}$, $\mathcal{N}'$ and $\mathcal{N}''$ are thick subcategories of $\T, \T'$ and $\T''$ respectively. Then (3.1) restricts to the following recollement $$\xymatrix{\N'\ar[r]^{i_*}&\N\ar@/^1pc/[l]^{i^!}\ar@/_1pc/[l]_{i^*}\ar[r]^{j^*} &\N''\ar@/^1pc/[l]^{j_*}\ar@/_1pc/[l]_{j_!} }$$
%if and only if (3.1) induces the following recollement $$\xymatrix{\T'/\N'\ar[r]^{\overline{i_*}}&\T/\N\ar@/^1pc/[l]^{\overline{i^!}}\ar@/_1.2pc/[l]_{\overline{i^*}}\ar[r]^{\overline{j^*}} &\T''/\N''\ar@/^1pc/[l]^{\overline{j_*}}\ar@/_1.2pc/[l]_{\overline{j_!}} }.$$
%\end{lem}

\section{2-recollement of singularity categories over triangular matrix algebras}\label{Bir}

In this section, $T=\left(
                                                                                 \begin{array}{cc}
                                                                                   A & M \\
                                                                                   0 & B \\
                                                                                 \end{array}
                                                                               \right)$ is a triangular matrix algebra with its corner algebras $A$ and $B$ Artinian and $_AM_B$ an $A$-$B$-bimodule.
Liu-Lu \cite{LL} and Zhang \cite{ZP} gave sufficient conditions for the existence of a recollement of $\D_{sg}(T)$ relative to $\D_{sg}(A)$ and $\D_{sg}(B)$.
In this section, we provide necessary and sufficient conditions for the existence of such recollement. Besides, we also give equivalent characterizations when there is a 2-recollement of  $\D_{sg}(T)$ relative to $\D_{sg}(A)$ and $\D_{sg}(B)$.
We first recall some basic definitions needed in the sequel.

Recall that a left $T$-module is identified with a triple $\left(
                                                                                        \begin{array}{c}
                                                                                          X \\
                                                                                          Y \\
                                                                                        \end{array}
                                                                                      \right)_\phi$,
 where $X\in \mod A$, $Y\in \mod B$ and $\phi:M\otimes_B Y\to X$ ia an $A$-morphism. If there is no possible confusion, we shall omit the morphism $\phi$ and write $\left(
                                                                                        \begin{array}{c}
                                                                                          X \\
                                                                                          Y \\
                                                                                        \end{array}
                                                                                      \right)$ for short.
 Analogously, a left $T$-module $\left(
\begin{array}{c}
 X \\
  Y \\
\end{array}
\right)_\phi$ is also identified with the triple $\left(
\begin{array}{c}
 X \\
  Y \\
\end{array}
\right)_{\widetilde{\phi}}$, where $\widetilde{\phi}:Y\to\Hom_A(M,X)$ is a $B$-morphism defined by $\widetilde{\phi}(y)(m)=\phi(m\otimes y)$ for any $m\in M$ and $y\in Y$.

                                                                                      %For example, we write $\left(
%                                                                                        \begin{array}{c}
%                                                                                          M\otimes_SY \\
%                                                                                          Y \\
%                                                                                        \end{array}
%                                                                                      \right)$ for the $T$-module $\left(
%                                                                                        \begin{array}{c}
%                                                                                          M\otimes_SY \\
%                                                                                          Y \\
%                                                                                        \end{array}
%                                                                                      \right)_1$.
%Given an $R$-map $\phi:M\otimes_SY\to X$, we denote $\widetilde{\phi}:Y\to \Hom_R(M,X)$
A $T$-morphism
$\left(
\begin{array}{c}
 X \\
  Y \\
\end{array}
\right)_\phi\to \left(
\begin{array}{c}
X' \\
Y' \\
\end{array}
\right)_{\phi'}$
will be identified with a pair
$\left(
\begin{array}{c}
f \\
g \\
\end{array}
\right)$,
where $f\in\Hom_A(X,X')$ and $g\in\Hom_B(Y,Y')$, such that the following diagram
$$\xymatrix{M\otimes_B Y\ar[r]^\phi\ar[d]_{1\otimes g} & X\ar[d]_f\\
M\otimes_B Y'\ar[r]^{\phi'} &X'
}$$ is commutative.

A sequence $0\to\left(
\begin{array}{c}
 X_1 \\
 Y_1 \\
\end{array}
\right)_{\phi_1}\xrightarrow{\left(
\begin{array}{c}
f_1 \\
g_1 \\
\end{array}
\right)}\left(
\begin{array}{c}
 X_2 \\
 Y_2 \\
\end{array}
\right)_{\phi_2}\xrightarrow{\left(
\begin{array}{c}
f_2 \\
g_2 \\
\end{array}
\right)}\left(
\begin{array}{c}
 X_3 \\
 Y_3 \\
\end{array}
\right)_{\phi_3}\to0$
in $\mod T$ is exact if and only if $0\to X_1\xrightarrow{f_1}X_2\xrightarrow{f_2}X_3\to0$ and $0\to Y_1\xrightarrow{g_1}Y_2\xrightarrow{g_2}Y_3\to0$ are exact in $\mod A$ and $\mod B$, respectively.
A $T$-module $\left(
                                  \begin{array}{c}
                                    X \\
                                    Y \\
                                  \end{array}
                                \right)_\phi$ is projective if and only if $Y\in\proj B$ and $\phi:M\otimes_BY\to X$ is an injective $A$-morphism with $\Coker\phi\in\proj A$.
We refer the reader to \cite[Section III.2]{ARS} for more details.

\vspace{0.2cm}
Let $e_A=\left(
                           \begin{array}{cc}
                             1 & 0 \\
                             0 & 0 \\
                           \end{array}
                         \right)$ and $e_B=\left(
                           \begin{array}{cc}
                             0 & 0 \\
                             0 & 1 \\
                           \end{array}
                         \right)$ be idempotents of $T$. It is known that $A\cong e_ATe_A\cong T/Te_BT$ and $B\cong e_BTe_B\cong T/Te_AT$ as algebras.
                         As a consequence of Example \ref{exa:2.8} (1), we have the following observation.

\begin{lem}\label{lem:3.1} We have the following 2-recollement of module categories:
$$\xymatrix@=3cm{\mod A\ar@/_3pc/[r]_{\Hom_A(e_AT,-)}\ar@/^1pc/[r]^{i_{e_A}}& \mod T\ar@/_3pc/[r]_{\Hom_T(B,-)}\ar@/^1pc/[l]^{S_{e_A}}\ar@/_3pc/[l]_{A\otimes_T-}\ar@/^1pc/[r]^{S_{e_B}} & \mod B\ar@/^1pc/[l]^{i_{e_B}}\ar@/_3pc/[l]_{Te_B\otimes_B-},}\eqno{(3.1)}$$
where $A\otimes_T\left(
                     \begin{array}{c}
                       X \\
                       Y \\
                     \end{array}
                   \right)_\phi\cong\Coker\phi$, $i_{e_A}(X)\cong\left(
                     \begin{array}{c}
                       X \\
                       0 \\
                     \end{array}
                   \right)$ and $\Hom_A(e_AT,X)\cong\left(
                     \begin{array}{c}
                       X \\
                       \Hom_A(M,X) \\
                     \end{array}
                   \right)_{\widetilde{id}}$;
while $Te_B\otimes_BY\cong\left(
                     \begin{array}{c}
                      M\otimes_BY \\
                       Y \\
                     \end{array}
                   \right)$, $i_{e_B}(Y)\cong\left(
                     \begin{array}{c}
                       0 \\
                       Y \\
                     \end{array}
                   \right)$ and $\Hom_T(B,\left(
                     \begin{array}{c}
                       X \\
                       Y \\
                     \end{array}
                   \right)_{\phi})\cong\Ker\widetilde{\phi}$.
\end{lem}

\begin{proof} To take $R=T$ and $e=e_B$ (resp. $e=e_A$) as in Example \ref{exa:2.8} (1), we have the following two recollements of module categories:
$$\xymatrix@=2.5cm{\mod A\ar[r]^{i_{e_A}}& \mod T\ar@/^1pc/[l]^{S_{e_A}}\ar@/_1.5pc/[l]_{A\otimes_T-}\ar[r]^{S_{e_B}} & \mod B\ar@/^1pc/[l]^{i_{e_B}}\ar@/_1.5pc/[l]_{Te_B\otimes_B-},}$$
$$\xymatrix@=2.5cm{\mod B\ar[r]^{i_{e_B}}& \mod T\ar@/^1pc/[l]^{\Hom_T(B,-)}\ar@/_1.5pc/[l]_{S_{e_B}}\ar[r]^{S_{e_A}} & \mod A\ar@/^1pc/[l]^{\Hom_A(e_AT,-)}\ar@/_1.5pc/[l]_{i_{e_A}}.}$$
To glue them together, we get the diagram (3.1). We will use the adjoint functors to get the expressions of the functors in (3.1). Indeed, let $X\in\mod A$ and $\left(\begin{array}{c}
                                                                                                                                                                                      X' \\
                                                                                                                                                                                      Y' \\
                                                                                                                                                                                    \end{array}
                                                                                                                                                                                  \right)_{\phi'}\in\mod T$.
We have $$\Hom_T(\left(
                     \begin{array}{c}
                       X' \\
                       Y' \\
                     \end{array}
                   \right)_{\phi'},\Hom_A(e_AT,X))\cong\Hom_A(S_{e_A}(\left(
                     \begin{array}{c}
                       X' \\
                       Y' \\
                     \end{array}
                   \right)_{\phi'}),X)\cong\Hom_A(X',X)$$
                   \hspace{1.2cm}$\cong\Hom_T(\left(
                     \begin{array}{c}
                       X' \\
                       Y' \\
                     \end{array}
                   \right)_{\widetilde{\phi'}},\left(
                     \begin{array}{c}
                       X \\
                       \Hom_A(M,X) \\
                     \end{array}
                   \right)_{\widetilde{id}})$.

                   \noindent By Yoneda Lemma, we get $\Hom_A(e_AT,X)\cong\left(
                     \begin{array}{c}
                       X \\
                       \Hom_A(M,X) \\
                     \end{array}
                   \right)_{\widetilde{id}}$. Similarly, one could obtain the expressions of another functors.
\end{proof}

Since the module category (over an arbitrary algebra) is a subcategory of its bounded derived category, it is natural to ask whether and when the 2-recollement (3.1) lifts to a 2-recollement of their bounded derived categories.
We provide the following necessary and sufficient conditions for this question.

\begin{lem}\label{lem:3.2} Let $T=\left(
                                                                                 \begin{array}{cc}
                                                                                   A & M \\
                                                                                   0 & B \\
                                                                                 \end{array}
                                                                               \right)$ be a triangular matrix algebra with $_AM_B$ an $A$-$B$-bimodule.
\begin{enumerate}
\item[(1)] (compare \cite{CL}) We have the following recollement of bounded derived categories: $$\xymatrix@=2.5cm{\D^b(\mod A)\ar[r]^{\D^b(i_{e_A})}& \D^b(\mod T)\ar@/^1pc/[l]^{\D^b(S_{e_A})}\ar@/_1.5pc/[l]_{A\otimes^\mathbb{L}_T-}\ar[r]^{\D^b(S_{e_B})} & \D^b(\mod B)\ar@/^1pc/[l]^{\D^b(i_{e_B})}\ar@/_1.5pc/[l]_{Te_B\otimes^\mathbb{L}_B-},}\eqno{(3.2)}$$ such that these six functors are the derived versions of those in (3.1) if and only if $\pd M_B<\infty$.
\item[(2)] We have the following recollement of bounded derived categories: $$\xymatrix@=2.5cm{\D^b(\mod B)\ar[r]^{\D^b(i_{e_B})}& \D^b(\mod T)\ar@/^1pc/[l]^{\mathbb{R}\Hom_T(B,-)}\ar@/_1.5pc/[l]_{\D^b(S_{e_B})}\ar[r]^{\D^b(S_{e_A})} & \D^b(\mod A)\ar@/^1pc/[l]^{\mathbb{R}\Hom_A(e_AT,-)}\ar@/_1.5pc/[l]_{\D^b(i_{e_A})},}\eqno{(3.3)}$$ such that these six functors are the derived versions of those in (3.1) if and only if $\pd_AM<\infty$.
\item[(3)] We have the following 2-recollement of bounded derived categories:
$$\xymatrix@=3cm{\D^b(\mod A)\ar@/_3pc/[r]_{\mathbb{R}\Hom_A(e_AT,-)}\ar@/^1pc/[r]^{\D^b(i_{e_A})}& \D^b(\mod T)\ar@/_3pc/[r]_{\mathbb{R}\Hom_T(B,-)}\ar@/^1pc/[l]^{\D^b(S_{e_A})}\ar@/_3pc/[l]_{A\otimes^\mathbb{L}_T-}\ar@/^1pc/[r]^{\D^b(S_{e_B})} & \D^b(\mod B)\ar@/^1pc/[l]^{\D^b(i_{e_B})}\ar@/_3pc/[l]_{Te_B\otimes^\mathbb{L}_B-}}\eqno{(3.4)}$$
if and only if $\pd_AM<\infty$ and $\pd M_B<\infty$.
\end{enumerate}
\end{lem}
\begin{proof} We only prove (2), (1) could be obtained by a similar argument (see also \cite[Theorem 2]{CL} for the proof of the ``if'' part) and (3) is a consequence of (1) and (2).
We proceed by letting $R=T$ and $e=e_A$ as in Example \ref{exa:2.8} (2). Note that $_TB\cong\left(
                                                                                                                                               \begin{array}{c}
                                                                                                                                                0 \\
                                                                                                                                                 B \\
                                                                                                                                               \end{array}
                                                                                                                                             \right)$ as left $T$-modules and $B_T\cong e_BT$ as right $T$-modules.
It follows that $B_T$ is projective and condition (iii) in Example \ref{exa:2.8} (2) follows. Notice that $\Ext^n_T(B,B)\cong\Ext^n_T(\left(
                                                                                                                                               \begin{array}{c}
                                                                                                                                                0 \\
                                                                                                                                                 B \\
                                                                                                                                               \end{array}
                                                                                                                                             \right)
,\left(
   \begin{array}{c}
     0 \\
     B \\
   \end{array}
 \right)
)\cong\Ext^n_B(B,B)=0$ for any $n\geq1$, where the last isomorphism could be easily checked by choosing a projective resolution of $\left(
                                                                                                                                  \begin{array}{c}
                                                                                                                                    0 \\
                                                                                                                                    B \\
                                                                                                                                  \end{array}
                                                                                                                                \right)$. So condition (i) in Example \ref{exa:2.8} (2) follows.
                                                                                                                                From \cite[Lemma 2.4]{LZHZ}, we know $\pd_TB<\infty$ if and only if $\pd_AM<\infty$.
%Now consider the following exact sequence of $T$-modules:
%$$0\to\left(
%        \begin{array}{c}
%          M \\
%          0 \\
%        \end{array}
%      \right)\to\left(
%                  \begin{array}{c}
%                    M \\
%                    B \\
%                  \end{array}
%                \right)\to\left(
%                            \begin{array}{c}
%                              0 \\
%                              B \\
%                            \end{array}
%                          \right)\to 0.$$
%Since $\left(
%                  \begin{array}{c}
%                    M \\
%                    B \\
%                  \end{array}
%                \right)$ is projective, one gets $\pd_T\left(
%                            \begin{array}{c}
%                              0 \\
%                              B \\
%                            \end{array}
%                          \right)<\infty$ if and only if $\pd_T\left(
%        \begin{array}{c}
%          M \\
%          0 \\
%        \end{array}
%      \right)<\infty$. From \cite[Lemma 2.3]{LZHZ}, we know $\pd_T\left(
%        \begin{array}{c}
%          M \\
%          0 \\
%        \end{array}
%      \right)=\pd_AM$. It follows that $\pd_TB<\infty$ if and only if $\pd_AM<\infty$.
Following Example \ref{exa:2.8} (2), we have the recollement (3.3) if and only if $\pd_AM<\infty$.
\end{proof}

%It is of interest to consider whether and when the recollements and 2-recollement in Lemma \ref{lem:3.2} restrict to the levels of subcategories of perfect and Gorenstein perfect complexes. We will study these questions in the following.
{\bf  Proof of Theorem \ref{thm:1.1}.} The proof of Theorem \ref{thm:1.1} is divided into the following two parts, including Proposition \ref{prop:3.3} and Proposition \ref{prop:3.5}.\qed

\begin{prop}\label{prop:3.3}  Let $T=\left(
                                                                                 \begin{array}{cc}
                                                                                   A & M \\
                                                                                   0 & B \\
                                                                                 \end{array}
                                                                               \right)$ be a triangular matrix algebra with $_AM_B$ an $A$-$B$-bimodule. Assume that $\pd M_B<\infty$. Then the following statements are equivalent.
\begin{enumerate}
\item[(1)] We have the following recollement of perfect complexes
$$\xymatrix@=2.5cm{\K^b(\proj A)\ar[r]^{\D^b(i_{e_A})}& \K^b(\proj T)\ar@/^1pc/[l]^{\D^b(S_{e_A})}\ar@/_1.5pc/[l]_{A\otimes^\mathbb{L}_T-}\ar[r]^{\D^b(S_{e_B})} & \K^b(\proj B)\ar@/^1pc/[l]^{\D^b(i_{e_B})}\ar@/_1.5pc/[l]_{Te_B\otimes^\mathbb{L}_B-}},\eqno{(3.5)}$$
where these six functors are the restrictions of those in (3.2).
\item[(2)] We have the following recollement of singularity categories
$$\xymatrix@=2.5cm{\D_{sg}(A)\ar[r]^{\overline{\D^b(i_{e_A})}}& \D_{sg}(T)\ar@/^1pc/[l]^{\overline{\D_{sg}(S_{e_A})}}\ar@/_1.5pc/[l]_{\overline{A\otimes^\mathbb{L}_T-}}\ar[r]^{\overline{\D_{sg}(S_{e_B})}} & \D_{sg}(B)\ar@/^1pc/[l]^{\overline{\D^b(i_{e_B})}}\ar@/_1.5pc/[l]_{\overline{Te_B\otimes^\mathbb{L}_B-}}},\eqno{(3.6)}$$ where these six functors are induced by those in (3.2).
\item[(3)] $\pd_A M<\infty$.
\end{enumerate}
\end{prop}

\begin{proof} Following Lemma \ref{lem:3.2} (1), we have the recollement (3.2) since $\pd M_B<\infty$.
 $(3)\Rightarrow(2)$ could be found in \cite[Theorem 3.2]{LL} and the equivalence of (1) and (2) follows directly from Lemma \ref{lem:2.7}.

$(1)\Rightarrow(3)$ Assume that we have the recollement (3.5), then $\D^b(i_{e_B})(B)\cong i_{e_B}(B)\in\K^b(\proj T)$. Since $i_{e_B}(B)\cong\left(
                                                                                                                                                           \begin{array}{c}
                                                                                                                                                             0 \\
                                                                                                                                                             B\\
                                                                                                                                                           \end{array}
                                                                                                                                                         \right)$ by Lemma \ref{lem:3.1}, it follows that $\pd_T\left(
                                                                                                                                                           \begin{array}{c}
                                                                                                                                                             0 \\
                                                                                                                                                             B\\
                                                                                                                                                           \end{array}
                                                                                                                                                         \right)<\infty$. Now consider the following exact sequence of $T$-modules:
$$0\to\left(
       \begin{array}{c}
         M \\
         0 \\
       \end{array}
     \right)\to\left(
       \begin{array}{c}
         M\\
         B \\
       \end{array}
     \right)\to\left(
       \begin{array}{c}
          0\\
         B \\
       \end{array}
     \right)\to0.\eqno{(ex1)}$$
Notice that $\left(
       \begin{array}{c}
         M\\
         B \\
       \end{array}
     \right)\in\proj T$ and $\pd_T\left(
       \begin{array}{c}
         0 \\
        B \\
       \end{array}
     \right)<\infty$, we have $\pd_T\left(
                                                                                                                                                                                              \begin{array}{c}
                                                                                                                                                                                                M \\
                                                                                                                                                                                                0 \\
                                                                                                                                                                                              \end{array}
                                                                                                                                                                                            \right)<\infty$.                                                                                                                                             %Thus we conclude $\pd_T\left(
%                                                                                                                                                                                              \begin{array}{c}
%                                                                                                                                                                                                M \\
%                                                                                                                                                                                                0 \\
%                                                                                                                                                                                              \end{array}
                                                                                                                                                                                           %\right)<\infty$ from the exact sequence $(ex2)$.
By \cite[Lemma 2.3]{LZHZ}, we get $\pd_A M<\infty$.
\end{proof}

\begin{rem}\label{rem:3.4} {\rm Assume that $\pd M_B<\infty$. Liu and Lu showed in \cite{LL} that if $\pd_A M<\infty$ then we have the recollement (3.6) of singularity categories. Whereas, by the equivalence of (2) and (3) in Proposition \ref{prop:3.3}, we know that ``$\pd_A M<\infty$'' is also a necessary condition for the existence of the recollement (3.6).
}\end{rem}

Similarly, we get the following

\begin{prop}\label{prop:3.5} Let $T=\left(
                                                                                 \begin{array}{cc}
                                                                                   A & M \\
                                                                                   0 & B \\
                                                                                 \end{array}
                                                                               \right)$ be a triangular matrix algebra with $_AM_B$ an $A$-$B$-bimodule. Assume that $\pd_A M<\infty$ and $M\in{^\bot A}$. Then the following statements are equivalent.
\begin{enumerate}
\item[(1)] We have the following recollement of perfect complexes
$$\xymatrix@=2.5cm{\K^b(\proj B)\ar[r]^{\D^b(i_{e_B})}& \K^b(\proj T)\ar@/^1pc/[l]^{\mathbb{R}\Hom_T(B,-)}\ar@/_1.5pc/[l]_{\D^b(S_{e_B})}\ar[r]^{\D^b(S_{e_A})} & \K^b(\proj A)\ar@/^1pc/[l]^{\mathbb{R}\Hom_A(e_AT,-)}\ar@/_1.5pc/[l]_{\D^b(i_{e_A})}},\eqno{(3.7)}$$ where these six functors are the restrictions of those in (3.3).
\item[(2)] We have the following recollement of singularity categories
$$\xymatrix@=2.5cm{\D_{sg}(B)\ar[r]^{\overline{\D^b(i_{e_B})}}& \D_{sg}(T)\ar@/^1pc/[l]^{\overline{\mathbb{R}\Hom_T(B,-)}}\ar@/_1.5pc/[l]_{\overline{\D^b(S_{e_B})}}\ar[r]^{\overline{\D^b(S_{e_A})}} & \D_{sg}(A)\ar@/^1pc/[l]^{\overline{\mathbb{R}\Hom_A(e_AT,-)}}\ar@/_1.5pc/[l]_{\overline{\D^b(i_{e_A})}}},\eqno{(3.8)}$$ where these six functors are induced by those in (3.3).
\item[(3)] $\pd_B\Hom_A(M,A)<\infty$.
\end{enumerate}
\end{prop}

\begin{proof} Following Lemma \ref{lem:3.2} (2), we have the recollement (3.3) since $\pd_A M<\infty$. In view of Lemma \ref{lem:2.7}, it suffices to show $\pd_B\Hom_A(M,A)<\infty$ if and only if the four functors $\D^b(S_{e_B})$, $\D^b(i_{e_B})$, $\D^b(S_{e_A})$ and $\mathbb{R}\Hom_A(e_AT,-)$ preserve prefect complexes.

Assume these four functors preserve prefect complexes.
Then we have that $\mathbb{R}\Hom_A(e_AT,A)\in\K^b(\proj T)$.
Since $M\in{^\perp A}$ by assumption and $e_AT\cong A\oplus M$ as $A$-modules, it follows that $\Ext^n_A(e_AT,A)=0$ for any $n\geq1$.
Hence $\Hom_A(e_AT,A)\cong\mathbb{R}\Hom_A(e_AT,A)\in\K^b(\proj T)$ and then $\pd_T\Hom_A(e_AT,A)<\infty$. Notice that $\Hom_A(e_AT,A)\cong\left(
                                                                                                                                             \begin{array}{c}
                                                                                                                                               A \\
                                                                                                                                               \Hom_A(M,A) \\
                                                                                                                                             \end{array}
                                                                                                                                           \right)_{\widetilde{id}}$ by Lemma \ref{lem:3.1}, we consider the following exact sequence of $T$-modules:
$$0\to\left(
        \begin{array}{c}
          A \\
          0 \\
        \end{array}
      \right)\to\left(
                                                                                                                                             \begin{array}{c}
                                                                                                                                               A \\
                                                                                                                                               \Hom_A(M,A) \\
                                                                                                                                             \end{array}
                                                                                                                                           \right)_{\widetilde{id}}\to\left(
                                                                                                                                                                        \begin{array}{c}
                                                                                                                                                                          0 \\
                                                                                                                                                                          \Hom_A(M,A) \\
                                                                                                                                                                        \end{array}
                                                                                                                                                                      \right)\to0.\eqno{(ex2)}$$
As $\left(
        \begin{array}{c}
          A \\
          0 \\
        \end{array}
      \right)$ is projective, we infer that $\pd_T\left(
                                                                                                                                                                        \begin{array}{c}
                                                                                                                                                                          0 \\
                                                                                                                                                                          \Hom_A(M,A) \\
                                                                                                                                                                        \end{array}
                                                                                                                                                                      \right)<\infty$ since $\pd_T\left(
                                                                                                                                             \begin{array}{c}
                                                                                                                                               A \\
                                                                                                                                               \Hom_A(M,A) \\
                                                                                                                                             \end{array}
                                                                                                                                           \right)_{\widetilde{id}}<\infty$. Thus we get $\pd_B\Hom_A(M,A)<\infty$ from \cite[Lemma 2.3]{LZHZ}.

Conversely, Assume $\pd_B\Hom_A(M,A)<\infty$. Since $S_{e_B}(T)\cong B$, one has $\D^b(S_{e_B})(\K^b(\proj T))\subseteq\K^b(\proj B)$. By Lemma \ref{lem:3.1}, we have that $i_{e_B}(B)\cong\left(
                                                                                                                                                                     \begin{array}{c}
                                                                                                                                                                       0 \\
                                                                                                                                                                       B \\
                                                                                                                                                                     \end{array}
                                                                                                                                                                   \right)$ and $S_{e_A}(T)\cong S_{e_A}(\left(
                                                                                                                                                                                                           \begin{array}{c}
                                                                                                                                                                                                             A \\
                                                                                                                                                                                                             0 \\
                                                                                                                                                                                                           \end{array}
                                                                                                                                                                                                         \right)\oplus\left(
                                                                                                                                                                                                                        \begin{array}{c}
                                                                                                                                                                                                                         M \\
                                                                                                                                                                                                                          B \\
                                                                                                                                                                                                                        \end{array}
                                                                                                                                                                                                                      \right)
                                                                                                                                                                   )\cong A\oplus M$. We infer that $S_{e_A}(T)\in\K^b(\proj A)$ since $\pd_AM<\infty$, and hence $\D^b(S_{e_A})(\K^b(\proj T))\subseteq\K^b(\proj A)$.
                                                                                                                                                                   %Now consider the following exact sequence of $T$-modules:
%$$0\to\left(
%       \begin{array}{c}
%         M \\
%         0 \\
%       \end{array}
%     \right)\to\left(
%       \begin{array}{c}
%         M\\
%         B \\
%       \end{array}
%     \right)\to\left(
%       \begin{array}{c}
%          0\\
%         B \\
%       \end{array}
%     \right)\to0.\eqno{(ex2)}$$
Notice that $\left(
       \begin{array}{c}
         M\\
         B \\
       \end{array}
     \right)\in\proj T$ and $\pd_T\left(
       \begin{array}{c}
         M \\
         0 \\
       \end{array}
     \right)<\infty$ from \cite[Lemma 2.3]{LZHZ}, we infer $\pd_Ti_{e_B}(B)<\infty$ from the exactness of the sequence $(ex1)$.  This implies that $i_{e_B}(B)\in\K^b(\proj T)$ and therefore $\D^b(i_{e_B})(\K^b(\proj B))\subseteq\K^b(\proj T)$.

Finally,  we will show $\mathbb{R}\Hom_A(e_AT,-)$ preserves perfect complexes to complete the proof. Since $M\in{^\bot A}$ by assumption, it follows that $\Ext_A^n(e_AT,A)=0$ for any $n\geq1$. Thus $\mathbb{R}\Hom_A(e_AT,A)\cong\Hom_A(e_AT,A)\cong\left(
                                                                                                                                             \begin{array}{c}
                                                                                                                                               A \\
                                                                                                                                               \Hom_A(M,A) \\
                                                                                                                                             \end{array}
                                                                                                                                           \right)_{\widetilde{id}}$.
As $\pd_B\Hom_A(M,A)<\infty$, we infer $\pd_T\left(
                                                                                                                                                                        \begin{array}{c}
                                                                                                                                                                          0 \\
                                                                                                                                                                          \Hom_A(M,A) \\
                                                                                                                                                                        \end{array}
                                                                                                                                                                      \right)<\infty$ from \cite[Lemma 2.3]{LZHZ}.
                                                                                                                                         Notice that $\left(
                                                                                                                                                        \begin{array}{c}
                                                                                                                                                          A\\
                                                                                                                                                          0 \\
                                                                                                                                                        \end{array}
                                                                                                                                                      \right)\in\proj T$, we conclude $\pd_T\left(
                                                                                                                                             \begin{array}{c}
                                                                                                                                               A \\
                                                                                                                                               \Hom_A(M,A) \\
                                                                                                                                             \end{array}
                                                                                                                                           \right)_{\widetilde{id}}<\infty$ by the exact sequence $(ex2)$. Hence $\mathbb{R}\Hom_A(e_AT,A)\in\K^b(\proj T)$ and then we have $\mathbb{R}\Hom_A(e_AT,-)$ preserves perfect complexes as desired.
\end{proof}

%We end this section with the proof of Theorem \ref{thm:1.1}.
%
%\vspace{2mm}
%{\bf The proof of Theorem \ref{thm:1.1}.}
%
%As a consequence of Propositions \ref{prop:3.3} and \ref{prop:3.5}, we get the following main result of this section.
%\begin{thm}\label{thm:3.6} Let $T=\left(
%                                                                                 \begin{array}{cc}
%                                                                                   A & M \\
%                                                                                   0 & B \\
%                                                                                 \end{array}
%                                                                               \right)$ be a triangular matrix algebra with $_AM_B$ an $A$-$B$-bimodule. Assume that $\pd_A M<\infty$, $\pd M_B<\infty$ and $M\in{^\bot A}$. Then we have the following 2-recollement of singularity categories:
%$$\xymatrix@=3cm{\D_{sg}(A)\ar@/_3pc/[r]_{\overline{\mathbb{R}\Hom_A(e_AT,-)}}\ar@/^1pc/[r]^{\overline{\D^b(i_{e_A})}}& \D_{sg}( T)\ar@/_3pc/[r]_{\overline{\mathbb{R}\Hom_T(B,-)}}\ar@/^1pc/[l]^{\overline{\D^b(S_{e_A})}}\ar@/_3pc/[l]_{\overline{A\otimes^\mathbb{L}_T-}}\ar@/^1pc/[r]^{\overline{\D^b(S_{e_B})}} & \D_{sg}( B)\ar@/^1pc/[l]^{\overline{\D^b(i_{e_B})}}\ar@/_3pc/[l]_{\overline{Te_B\otimes^\mathbb{L}_B-}}}\eqno{(3.9)}$$
%if and only if $\pd_B\Hom_A(M,A)<\infty$.
%\end{thm}

\section{2-recollement of Gorenstein defect categories over triangular matrix algebras}\label{Gd}

In this section, we take $T=\left(
                                                                                 \begin{array}{cc}
                                                                                   A & M \\
                                                                                   0 & B \\
                                                                                 \end{array}
                                                                               \right)$ the triangular matrix algebra as that in Section \ref{Bir}. As a successor, we will study 2-recollement of the Gorenstein defect category $\D_{def}(T)$ relative to $\D_{def}(A)$ and $\D_{def}(B)$.

\vspace{3mm}
To prove Theorem \ref{thm:1.2}, we need some preparations. Firstly, we should know the concrete form of Gorenstein projective modules over $T$.
%\vspace{2mm}
%Recall from \cite{ZP} that $_AM_B$ is {\it compatible} if $M\otimes_B-$ carries every acyclic complex of projective $B$-modules to acyclic $A$-complex and $M\in(\Gproj A)^\perp$.
%If $_AM_B$ is compatible, it is not hard to see $\Tor_i^B(M,G)=0$ for any $G\in\Gproj B$ and $i\geq1$.

\begin{lem}(see \cite[Theorem 1.4]{ZP})\label{lem:4.1} Let $T=\left(
                                                                                 \begin{array}{cc}
                                                                                   A & M \\
                                                                                   0 & B \\
                                                                                 \end{array}
                                                                               \right)$ be a triangular matrix algebra with $_AM_B$ compatible.
Then $\left(
        \begin{array}{c}
          X \\
          Y \\
        \end{array}
      \right)_\phi\in\Gproj T$ if and only if $Y\in\Gproj B$ and $\phi:M\otimes_B Y\to X$ is an injective $A$-morphism with $\Coker\phi\in\Gproj A$.
\end{lem}

%\begin{df} {\rm An $A$-$B$-bimodule $_AM_B$ is said to be {\it left Gorenstein singular} if $\Gpd_B\Hom_A(M,F)<\infty$ for any $F\in\Gproj A$; while $_AM_B$ is said to be {\it right Gorenstein singular} if $\Gpd_AM\otimes_BG<\infty$ for any $G\in\Gproj B$. We call $_AM_B$ {\it Gorenstein singular} if it is both left and right Gorenstein singular.
%}\end{df}

We also need the following fact.

\begin{lem}\label{lem:4.3}(see \cite[Corollary 4.2]{LHZ}) Let $T=\left(
                                                                                 \begin{array}{cc}
                                                                                   A & M \\
                                                                                   0 & B \\
                                                                                 \end{array}
                                                                               \right)$ be a triangular matrix algebra with $_AM_B$ compatible. The following statements hold true.
\begin{enumerate}
\item[(1)] $\Gpd_T\left(
                                                                                            \begin{array}{c}
                                                                                              X \\
                                                                                              0 \\
                                                                                            \end{array}
                                                                                          \right)
                                                                               =\Gpd_AX$.
\item[(2)] Assume that $M$ is right Gorenstein singular.
                                                                                                                                                        Then $\Gpd_T\left(
                                                                                                                                                                      \begin{array}{c}
                                                                                                                                                                        0 \\
                                                                                                                                                                        Y\\
                                                                                                                                                                      \end{array}
                                                                                                                                                                    \right)
                                                                                                                                                        <\infty$ if and only if $\Gpd_BY<\infty$.
\end{enumerate}
\end{lem}

Note that the 2-recollement (1.2) in Theorem \ref{thm:1.2} is consisted of two recollements.  We will give equivalent characterizations for the existence of each one.

\begin{prop}\label{prop:4.4} (compare \cite[Theorem 1.2]{LZHZ} and \cite[Theorem 3.12]{L}) Let $T=\left(
                                                                                 \begin{array}{cc}
                                                                                   A & M \\
                                                                                   0 & B \\
                                                                                 \end{array}
                                                                               \right)$ be a triangular matrix algebra with $_AM_B$ compatible. Assume that $\pd M_B<\infty$. Then the following statements are equivalent.
\begin{enumerate}
\item[(1)]  We have the following recollement of Gorenstein perfect complexes
$$\xymatrix@=2.5cm{\Gperf(A)\ar[r]^{\D^b(i_{e_A})}& \Gperf(T)\ar@/^1pc/[l]^{\D^b(S_{e_A})}\ar@/_1.5pc/[l]_{A\otimes^\mathbb{L}_T-}\ar[r]^{\D^b(S_{e_B})} & \Gperf(B) \ar@/^1pc/[l]^{\D^b(i_{e_B})}\ar@/_1.5pc/[l]_{Te_B\otimes^\mathbb{L}_B-}},\eqno{(4.1)}$$ where these six functors are the restrictions of those in (3.2).

\item[(2)] We have the following recollement of Gorenstein defect categories
$$\xymatrix@=2.5cm{\D_{def}(A)\ar[r]^{\widetilde{\D^b(i_{e_A})}}& \D_{def}(T)\ar@/^1pc/[l]^{\widetilde{\D^b(S_{e_A})}}\ar@/_1.5pc/[l]_{\widetilde{A\otimes^\mathbb{L}_T-}}\ar[r]^{\widetilde{\D^b(S_{e_B})}} & \D_{def}( B)\ar@/^1pc/[l]^{\widetilde{\D^b(i_{e_B})}}\ar@/_1.5pc/[l]_{\widetilde{Te_B\otimes^\mathbb{L}_B-}}},\eqno{(4.2)}$$ where these six functors are induced by those in (3.2).

\item[(3)] $M$ is right Gorenstein singular.
\end{enumerate}
%In this case, we have the following recollement of stable category of Gorenstein projective modules
%$$\xymatrix{\underline{\Gproj A}\ar[r]& \underline{\Gproj T}\ar@/^1pc/[l]\ar@/_1pc/[l]\ar[r] & \underline{\Gproj B}\ar@/^1pc/[l]\ar@/_1pc/[l]},\eqno{(4.6)}$$ provided that $\pd_AM<\infty$,
%where these six functors are induced by those in (4.1).
\end{prop}

\begin{proof} Following Lemma \ref{lem:3.2} (1), we have the recollement (3.2) since $\pd M_B<\infty$. The equivalence of (1) and (2) follows from Lemmas \ref{lem:2.5} and \ref{lem:2.7}, and $(3)\Rightarrow(2)$ could be found in \cite[Theorem 1.2]{LZHZ}.

$(1)\Rightarrow(3)$ Assume that we have the recollement (4.1). Take any $Y\in\Gproj B$, we obtain $\D^b(i_{e_B})(Y)\cong i_{e_B}(Y)$ is Gorenstein perfect. Notice that $i_{e_B}(Y)\cong\left(
                                                                                                                                                                                     \begin{array}{c}
                                                                                                                                                                                       0 \\
                                                                                                                                                                                       Y \\
                                                                                                                                                                                     \end{array}
                                                                                                                                                                                   \right)$,
it follows that $\Gpd_T\left(
                                                                                                                                                                                     \begin{array}{c}
                                                                                                                                                                                       0 \\
                                                                                                                                                                                       Y \\
                                                                                                                                                                                     \end{array}
                                                                                                                                                                                   \right)<\infty$.
Consider the following exact sequence of $T$-modules:
$$0\to\left(
       \begin{array}{c}
          M\otimes_BY \\
         0 \\
       \end{array}
     \right)\to\left(
       \begin{array}{c}
         M\otimes_BY\\
         Y \\
       \end{array}
     \right)\to\left(
       \begin{array}{c}
          0\\
         Y \\
       \end{array}
     \right)\to0.$$
Since $\left(
       \begin{array}{c}
         M\otimes_BY\\
         Y \\
       \end{array}
     \right)\in\Gproj T$ by Lemma \ref{lem:4.1} and $\Gpd_T\left(
                                                                                                                                                                                     \begin{array}{c}
                                                                                                                                                                                       0 \\
                                                                                                                                                                                       Y \\
                                                                                                                                                                                     \end{array}
                                                                                                                                                                                   \right)<\infty$, we get $\Gpd_T\left(
       \begin{array}{c}
          M\otimes_BY \\
         0 \\
       \end{array}
     \right)<\infty$. Hence we conclude $\Gpd_AM\otimes_BY<\infty$ from Lemma \ref{lem:4.3}, this means $M$ is right Gorenstein singular.
%If $\pd_AM<\infty$, we obtain the recollement (3.7) by Proposition \ref{prop:3.3}. Combining recollements (4.1) and (3.7) with Lemmas \ref{lem:2.5} and \ref{lem:2.7} we get the recollement (4.6).
\end{proof}

Viewing the stable category of Gorenstein projective modules as a triangulated subcategory of the singularity category, we get the following

\begin{cor}\label{cor:4.5} Suppose that $_AM_B$ has finite projective dimension both as a left $A$- and right $B$-module. Then we have the following recollement of stable categories of Gorenstein projective modules
$$\xymatrix{\underline{\Gproj A}\ar[r]& \underline{\Gproj T}\ar@/^1pc/[l]\ar@/_1pc/[l]\ar[r] & \underline{\Gproj B}\ar@/^1pc/[l]\ar@/_1pc/[l]}\eqno{(4.3)}$$
such that all these functors are the restrictions of those in (3.6) if and only if $M$ is right Gorenstein singular.
\end{cor}
\begin{proof} Since $_AM_B$ has finite projective dimension both as a left $A$- and right $B$-module, it is not hard to see $_AM_B$ is compatible. By Proposition \ref{prop:3.3}, we get the recollement (3.6).

For the ``if'' part, assume $M$ is right Gorenstein singular. From Proposition \ref{prop:4.4}, we get the recollement (4.2), which is also induced by the recollement (3.6) since their functors are initially from (3.2). Therefore, we have the recollement (4.3) from Lemma \ref{lem:2.7}.

Conversely, assume that we have the recollement (4.3).
 From Lemma \ref{lem:2.7}, we have the recollement (4.2). Then we infer that $M$ is right Gorenstein singular from Proposition \ref{prop:4.4}.
\end{proof}

\begin{rem} {\rm The recollement (4.3) of stable categories of Gorenstein projective modules has been considered by Zhang \cite{ZP}, where he proves that if $T$ is Gorenstein and $_AM$ is projective then we have the recollement (4.3) (see \cite[Theorem 3.5]{ZP}). Whereas, if $T$ is Gorenstein and $_AM$ is projective, it is not hard to see (combine \cite[Theorem 2.2]{ZP}) all the conditions in Corollary \ref{cor:4.5} are satisfied and $M$ is right Gorenstein singular. Therefore, our result generalizes Zhang's to a more general case. Besides, our proofs are quite different.
}\end{rem}

\begin{exa} {\rm Let $k$ be a field and $Q$ the following quiver:
$$\xymatrix{1\ar@/^/[r]^\alpha &2\ar@/^/[l]^{\alpha'}&\\
3\ar[u]^{\gamma} \ar@/^/[r]^{\beta}&4\ar[u]_{\delta}\ar[r]^{\theta}\ar@/^/[l]^{\beta'} &5.
}$$
Consider the $k$-algebra $T=kQ/I$, where $I$ is generated by $\alpha'\alpha$, $\alpha\alpha'$, $\beta'\beta$, $\beta\beta'$, $\theta\beta$, $\alpha\gamma-\delta\beta$ and $\alpha'\delta-\gamma\beta'$. Let $e_i$ be the idempotent corresponding to the vertex $i$ and put $e=e_1+e_2$.
Denote by $A=eTe$ and $B=(1-e)T(1-e)$. It follows that $T=\left(
                                                                                 \begin{array}{cc}
                                                                                   A & M \\
                                                                                   0 & B \\
                                                                                 \end{array}
                                                                               \right)$ with $M=eT(1-e)$. It is easy to check $A$ is self-injective, then $\D_{sg}(A)\simeq\underline{\Gproj A}\simeq\underline{\mod A}$ and hence $\D_{def}(A)$ vanishes.
Since $B$ is of radical square zero but not self-injective, we infer from \cite{C2'} that $B$ is CM-free (that is $\Gproj B=\proj B$). Hence we obtain $\underline{\Gproj B}$ vanishes and $\D_{def}(B)=\D_{sg}(B)$. Notice that $_AM$ and $M_B$ are projective, and $M$ is right Gorenstein singular since $A$ is self-injective.
Following Propositions \ref{prop:3.3}, \ref{prop:4.4} and Corollary \ref{cor:4.5},
we get the following recollement of singularity categories
$$\xymatrix{\D_{sg}(A)\ar[r]& \D_{sg}(T)\ar@/^1pc/[l]\ar@/_1pc/[l]\ar[r] & \D_{sg}(B)\ar@/^1pc/[l]\ar@/_1pc/[l]}$$
and triangle-equivalences $\D_{def}(T)\simeq\D_{def}(B)=\D_{sg}(B)$ and $\underline{\Gproj T}\simeq\underline{\Gproj A}\simeq\D_{sg}(A)\simeq\underline{\mod A}$.
%To conclude, we get the following recollement: $$\xymatrix{\underline{\Gproj T}\ar[r]& \D_{sg}(T)\ar@/^1pc/[l]\ar@/_1pc/[l]\ar[r] & \D_{def}(T)\ar@/^1pc/[l]\ar@/_1pc/[l]}.$$
}\end{exa}

Let $R$ be an Artin algebra and $X^\bullet$ a complex of $R$-modules. The {\it length} $l(X^\bullet)$ of $X^\bullet$ is defined to be the cardinal of the set $\{X^i\neq0|i\in\mathbb{Z}\}$. Let $n\in\mathbb{Z}$, denote by $X^\bullet_{\geqslant n}$ the complex with the $i$th component equal to $X^i$ whenever $i\geqslant n$ and to 0 elsewhere.

\begin{prop}\label{prop:4.7} Let $T=\left(
                                                                                 \begin{array}{cc}
                                                                                   A & M \\
                                                                                   0 & B \\
                                                                                 \end{array}
                                                                               \right)$ be a triangular matrix algebra with $_AM_B$ compatible. Assume that $\pd_A M<\infty$ and $M\in{^\bot\Gproj A}$. Then the following statements are equivalent.
\begin{enumerate}
\item[(1)] We have the following recollement of Gorenstein perfect complexes
$$\xymatrix@=2.5cm{\Gperf(B)\ar[r]^{\D^b(i_{e_B})}& \Gperf(T)\ar@/^1pc/[l]^{\mathbb{R}\Hom_T(B,-)}\ar@/_1.5pc/[l]_{\D^b(S_{e_B})}\ar[r]^{\D^b(S_{e_A})} & \Gperf(A) \ar@/^1pc/[l]^{\mathbb{R}\Hom_A(e_AT,-)}\ar@/_1.5pc/[l]_{\D^b(i_{e_A})}},\eqno{(4.4)}$$ where these six functors are the restrictions of those in (3.3).

\item[(2)] We have the following recollement of Gorenstein defect categories
$$\xymatrix@=2.5cm{\D_{def}(B)\ar[r]^{\widetilde{\D^b(i_{e_B})}}& \D_{def}(T)\ar@/^1pc/[l]^{\widetilde{\mathbb{R}\Hom_T(B,-)}}\ar@/_1.5pc/[l]_{\widetilde{\D^b(S_{e_B})}}\ar[r]^{\widetilde{\D^b(S_{e_A})}} & \D_{def}( A)\ar@/^1pc/[l]^{\widetilde{\mathbb{R}\Hom_A(e_AT,-)}}\ar@/_1.5pc/[l]_{\widetilde{\D^b(i_{e_A})}}},\eqno{(4.5)}$$ where these six functors are induced by those in (3.3).

\item[(3)] $M$ is Gorenstein singular.
\end{enumerate}
%In this case, we have the following recollement of stable category of Gorenstein projective modules
%$$\xymatrix{\underline{\Gproj B}\ar[r]& \underline{\Gproj T}\ar@/^1pc/[l]\ar@/_1pc/[l]\ar[r] & \underline{\Gproj A}\ar@/^1pc/[l]\ar@/_1pc/[l]},\eqno{(4.6)}$$ provided that $\pd_B\Hom_A(M,A)<\infty$, where these six functors are induced by those in (4.4).
\end{prop}

\begin{proof} Following Lemma \ref{lem:3.2} (2), we have the recollement (3.3) since $\pd_A M<\infty$. In view of Lemmas \ref{lem:2.5} and \ref{lem:2.7}, it suffices to show $M$ is Gorenstein singular if and only if the four functors $\D^b(S_{e_B})$, $\D^b(i_{e_B})$, $\D^b(S_{e_A})$ and $\mathbb{R}\Hom_A(e_AT,-)$ preserve Gorenstein perfect complexes.

Assume these four functors preserve Gorenstein perfect complexes. Take any $Y\in\Gproj B$, from Lemma \ref{lem:4.1} we know $\left(
                                                                                                                             \begin{array}{c}
                                                                                                                               M\otimes_BY \\
                                                                                                                               Y \\
                                                                                                                             \end{array}
                                                                                                                           \right)\in\Gproj T$ and then it is Gorenstein perfect. It follows that $$\D^b(S_{e_A})(\left(
                                                                                                                             \begin{array}{c}
                                                                                                                               M\otimes_BY \\
                                                                                                                               Y \\
                                                                                                                             \end{array}
                                                                                                                           \right))\cong S_{e_A}(\left(
                                                                                                                             \begin{array}{c}
                                                                                                                               M\otimes_BY \\
                                                                                                                               Y \\
                                                                                                                             \end{array}
                                                                                                                           \right))\cong M\otimes_BY\in\Gperf(A),$$
this implies $\Gpd_AM\otimes_BY<\infty$. Hence $M$ is right Gorenstein singular. Notice that $e_AT\cong A\oplus M$, we infer $e_AT\in{^\bot\Gproj A}$ since
$M\in{^\bot\Gproj A}$.
For any $X\in\Gproj A$, it follows that $$\mathbb{R}\Hom_A(e_AT,X)\cong\Hom_A(e_AT,X)\cong\left(
                                                                                            \begin{array}{c}
                                                                                              X \\
                                                                                              \Hom_A(M,X) \\
                                                                                            \end{array}
                                                                                          \right)_{\widetilde{id}}.$$
Since $\mathbb{R}\Hom_A(e_AT,-)$ preserves Gorenstein perfect complexes, we get $\Gpd_T\left(
                                                                                            \begin{array}{c}
                                                                                              X \\
                                                                                              \Hom_A(M,X) \\
                                                                                            \end{array}
                                                                                          \right)_{\widetilde{id}}<\infty$. Consider the following exact sequence of $T$-modules:
                                                                                          $$0\to\left(
                                                                                            \begin{array}{c}
                                                                                              X \\
                                                                                              0 \\
                                                                                            \end{array}
                                                                                          \right)\to\left(
                                                                                            \begin{array}{c}
                                                                                              X \\
                                                                                              \Hom_A(M,X) \\
                                                                                            \end{array}
                                                                                          \right)_{\widetilde{id}}\to\left(
                                                                                            \begin{array}{c}
                                                                                              0 \\
                                                                                              \Hom_A(M,X) \\
                                                                                            \end{array}
                                                                                          \right)\to0.\eqno{(ex3)}$$
                                                                                         Notice that $\left(
                                                                                            \begin{array}{c}
                                                                                              X \\
                                                                                              0 \\
                                                                                            \end{array}
                                                                                          \right)\in\Gproj T$ by Lemma \ref{lem:4.1}, one has $\Gpd_T\left(
                                                                                            \begin{array}{c}
                                                                                              0 \\
                                                                                              \Hom_A(M,X) \\
                                                                                            \end{array}
                                                                                          \right)<\infty$.
By Lemma \ref{lem:4.3} (2), we obtain $\Gpd_B\Hom_A(M,X)<\infty$ and then $M$ is left Gorenstein singular.  To sum up, we get that $M$ is Gorenstein singular.
%(every Gorenstein projective resolution of $\left(\begin{array}{c}0 \\\Hom_A(M,X) \\\end{array}\right)$ of finite length induces that of $\Hom_A(M,X)$).

Conversely, assume that $M$ is Gorenstein singular. Since $S_{e_B}$ preserves Gorenstein projective modules by Lemma \ref{lem:3.1}, it is easy to check that $\D^b(S_{e_B})$ preserves Gorenstein perfect complexes.
For any $Y\in\Gproj B$, by Lemma \ref{lem:3.1}, we obtain $i_{e_B}(Y)\cong\left(
                                                                                                                                                                                                   \begin{array}{c}
                                                                                                                                                                                                     0 \\
                                                                                                                                                                                                     Y \\
                                                                                                                                                                                                   \end{array}
                                                                                                                                                                                                 \right)$.
Since $M$ is right Gorenstein singular, we infer $\Gpd_Bi_{e_B}(Y)<\infty$ from Lemma \ref{lem:4.3} (2).  Let $Y^\bullet$ be a bounded complex of Gorenstein $B$-modules. Since $i_{e_B}$ is exact, we have $\D^b(i_{e_B})(Y^\bullet)\cong i_{e_B}(Y^\bullet)$, it is a bounded complex with each degree being of finite Gorenstein projective dimension. It follows from Lemma \ref{lem:2.4} that $\D^b(i_{e_B})(Y^\bullet)\in\Gperf(T)$. Now for any  $\left(
               \begin{array}{c}
                 F \\
                 G \\
               \end{array}
             \right)_\phi\in\Gproj T$, by Lemma \ref{lem:3.1} we have $S_{e_A}(\left(
               \begin{array}{c}
                 F \\
                 G \\
               \end{array}
             \right)_\phi)\cong F$.
Following Lemma \ref{lem:4.1}, we have the following exact sequence of $A$-modules
             $$0\to M\otimes_BG\to F\to \Coker\phi\to0$$  with $G\in\Gproj B$ and $\Coker\phi\in\Gproj A$.
Then $\Gpd_AM\otimes_BG<\infty$ since $M$ is right Gorenstein singular and hence $\Gpd_AF<\infty$. Similarly as above, we conclude that  $\D^b(S_{e_A})$ preserves Gorenstein perfect complexes.

Finally, it remains to show $\mathbb{R}\Hom_A(e_AT,-)$ preserves Gorenstein perfect complexes to complete our proof. To do this, let $X^\bullet$ be a bounded complex of Gorenstein projective $A$-modules. We proceed by induction on length $l(X^\bullet)$ of $X^\bullet$. If $l(X^\bullet)=1$, we may suppose $X^\bullet=X$ is the stalk complex concentrated in degree 0. Since $M\in{^\bot\Gproj A}$, it is clear to see  $$\mathbb{R}\Hom_A(e_AT,X)\cong\Hom_A(e_AT,X)\cong\left(
                                                             \begin{array}{c}
                                                               X \\
                                                               \Hom_A(M,X) \\
                                                             \end{array}
                                                           \right)_{\widetilde{id}}.$$
As $X\in\Gproj A$, one has $\left(
                              \begin{array}{c}
                               X \\
                                0 \\
                              \end{array}
                            \right)\in\Gproj T$.  Besides, we have $\Gpd_B\Hom_A(M,X)<\infty$ since $M$ is left Gorenstein singular. From Lemma \ref{lem:4.3} (2), we know $\Gpd_T\left(
                  \begin{array}{c}
                    0 \\
                    \Hom_A(M,X) \\
                  \end{array}
                \right)<\infty$.
Hence from the exactness of the sequence $(ex3)$, we get $\Gpd_T\left(
                                                             \begin{array}{c}
                                                               X \\
                                                               \Hom_A(M,X) \\
                                                             \end{array}
                                                           \right)_{\widetilde{id}}<\infty$ and then $\mathbb{R}\Hom_A(e_AT,X)\in\Gperf(T)$.

Now suppose $l(X^\bullet)=n\geq2$ and the assertion holds true for any integer less than $n$. We may assume $X^\bullet=0\to X^{0}\to X^{1}\to\cdots\to X^{n-1}\to0.$
It induces a triangle $$X^{0}[-1]\to X^\bullet_{\geq 1}\to X^\bullet\to X^{0}$$ in $\D^b(\mod A)$.
Apply the functor $\mathbb{R}\Hom_A(e_AT,-)$ to it, we get the following triangle
$$\mathbb{R}\Hom_A(e_AT,X^{0})[-1]\to \mathbb{R}\Hom_A(e_AT,X^\bullet_{\geq 1})\to \mathbb{R}\Hom_A(e_AT,X^\bullet)\to\mathbb{R}\Hom_A(e_AT,X^0)$$ in $\D^b(\mod T)$.
By the induction hypothesis, we have both $\mathbb{R}\Hom_A(e_AT,X^{0})$ and $\mathbb{R}\Hom_A(e_AT,X^\bullet_{\geq 1})$ are Gorenstein perfect. Note that $\Gperf(T)$ is a thick subcategory of $\D^b(\mod T)$. Hence we obtain $\mathbb{R}\Hom_A(e_AT,X^\bullet)\in\Gperf(T)$, that is, $\mathbb{R}\Hom_A(e_AT,-)$ preserves Gorenstein perfect complexes.
\end{proof}

Combine Proposition \ref{prop:4.7} with Proposition \ref{prop:3.5}, we get the following equivalent characterizations for the existence of a recollement of $\underline{\Gproj T}$ relative to $\underline{\Gproj B}$ and $\underline{\Gproj A}$.

\begin{cor}\label{cor:4.8}  Let $T=\left(
                                                                                 \begin{array}{cc}
                                                                                   A & M \\
                                                                                   0 & B \\
                                                                                 \end{array}
                                                                               \right)$ be a triangular matrix algebra with $_AM_B$ compatible. Assume that $\pd_A M<\infty$, $\pd_B\Hom_A(M,A)<\infty$ and $M\in{^\bot\Gproj A}$.  We have the following recollement of stable category of Gorenstein projective modules
$$\xymatrix{\underline{\Gproj B}\ar[r]& \underline{\Gproj T}\ar@/^1pc/[l]\ar@/_1pc/[l]\ar[r] & \underline{\Gproj A}\ar@/^1pc/[l]\ar@/_1pc/[l]}\eqno{(4.6)}$$ such that all these six functors are the restrictions of those in (3.8) if and only if $M$ is Gorenstein singular.
\end{cor}
\begin{proof} %By Proposition \ref{prop:4.7}, we have the recollement (4.4). Since $\pd_B\Hom_A(M,A)<\infty$, in view of Proposition \ref{prop:3.5}, we have the recollement (3.7). To combine (4.4) with (3.7), we infer the following recollement by Lemma \ref{lem:2.7} $$\xymatrix{\Gperf(B)/\K^b(\proj B)\ar[r]& \Gperf(T)/\K^b(\proj T)\ar@/^1pc/[l]\ar@/_1pc/[l]\ar[r] & \Gperf(A)/\K^b(\proj A)\ar@/^1pc/[l]\ar@/_1pc/[l]}.$$
%Therefore, we get the recollement (4.3) by Lemma \ref{lem:2.5}.
The proof is similar as that in Corollary \ref{cor:4.5}, we omit it.
\end{proof}

{\bf Proof of Theorem \ref{thm:1.2}.} This follows directly from Propositions \ref{prop:4.4} and \ref{prop:4.7}.\qed

\vspace{2mm}
Consequently, we have the following commutative diagram such that all the functors involving in the recollements and 2-recollements are mentioned above.

\begin{cor}\label{cor:4.9} Let $T=\left(
                                                                                 \begin{array}{cc}
                                                                                   A & M \\
                                                                                   0 & B \\
                                                                                 \end{array}
                                                                               \right)$ be a triangular matrix algebra with $_AM_B$ compatible.
\begin{enumerate}
\item[(1)] Assume that $\pd M_B<\infty$. Then we have the following commutative diagram of recollements
$$\xymatrix@R=4.5ex{0\ar[d]&0\ar[d]&0\ar[d]\\
 \underline{\Gproj A}\ar[d]\ar[r]& \underline{\Gproj T}\ar[d]\ar@/^1pc/[l]\ar@/_1pc/[l]\ar[r] & \underline{\Gproj B}\ar[d]\ar@/^1pc/[l]\ar@/_1pc/[l]\\
\D_{sg}(A)\ar[d]\ar[r]& \D_{sg}(T)\ar[d]\ar@/^1pc/[l]\ar@/_1pc/[l]\ar[r] & \D_{sg}(B)\ar[d]\ar@/^1pc/[l]\ar@/_1pc/[l]\\
\D_{def}(A)\ar[d]\ar[r]& \D_{def}(T)\ar[d]\ar@/^1pc/[l]\ar@/_1pc/[l]\ar[r] & \D_{def}(B)\ar[d]\ar@/^1pc/[l]\ar@/_1pc/[l]\\
0&0&0}$$
if and only if $\pd_AM<\infty$ and $M$ is right Gorenstein singular.%, where all the functors are those mentioned above.
\item[(2)] Assume that $\pd_A M<\infty$ and $M\in{^\bot\Gproj A}$. Then we have the following commutative diagram of recollements
$$\xymatrix@R=4.5ex{0\ar[d]&0\ar[d]&0\ar[d]\\
 \underline{\Gproj B}\ar[d]\ar[r]& \underline{\Gproj T}\ar[d]\ar@/^1pc/[l]\ar@/_1pc/[l]\ar[r] & \underline{\Gproj A}\ar[d]\ar@/^1pc/[l]\ar@/_1pc/[l]\\
\D_{sg}(B)\ar[d]\ar[r]& \D_{sg}(T)\ar[d]\ar@/^1pc/[l]\ar@/_1pc/[l]\ar[r] & \D_{sg}(A)\ar[d]\ar@/^1pc/[l]\ar@/_1pc/[l]\\
\D_{def}(B)\ar[d]\ar[r]& \D_{def}(T)\ar[d]\ar@/^1pc/[l]\ar@/_1pc/[l]\ar[r] & \D_{def}(A)\ar[d]\ar@/^1pc/[l]\ar@/_1pc/[l]\\
0&0&0}$$ if and only if $\pd_B\Hom_A(M,A)<\infty$ and $M$ is Gorenstein singular.%, where all the functors are those mentioned above.
\item[(3)] Assume that $\pd_A M<\infty$, $M\in{^\bot\Gproj A}$ and  $\pd M_B<\infty$. Then we have the following commutative diagram of 2-recollements
$$\xymatrix@R=4.5ex{0\ar[d]&0\ar[d]&0\ar[d]\\
 \underline{\Gproj A}\ar[d]\ar@/^0.35pc/[r]\ar@/_1.05pc/[r]& \underline{\Gproj T}\ar[d]\ar@/^0.35pc/[l]\ar@/_1.05pc/[l]\ar@/^0.35pc/[r]\ar@/_1.05pc/[r] & \underline{\Gproj B}\ar[d]\ar@/^0.35pc/[l]\ar@/_1.05pc/[l]\\
\D_{sg}(A)\ar[d]\ar@/^0.35pc/[r]\ar@/_1.05pc/[r]& \D_{sg}(T)\ar[d]\ar@/^0.35pc/[l]\ar@/_1.05pc/[l]\ar@/^0.35pc/[r]\ar@/_1.05pc/[r] & \D_{sg}(B)\ar[d]\ar@/^0.35pc/[l]\ar@/_1.05pc/[l]\\
\D_{def}(A)\ar[d]\ar@/^0.35pc/[r]\ar@/_1.05pc/[r]& \D_{def}(T)\ar[d]\ar@/^0.35pc/[l]\ar@/_1.05pc/[l]\ar@/^0.35pc/[r]\ar@/_1.05pc/[r] & \D_{def}(B)\ar[d]\ar@/^0.35pc/[l]\ar@/_1.05pc/[l]\\
0&0&0}$$
if and only if $\pd_B\Hom_A(M,A)<\infty$ and $M$ is Gorenstein singular.%, where all the functors are those mentioned above.
\end{enumerate}
\end{cor}
%\begin{proof} The commutativity of these diagrams are clear since every vertical functor is canonical and every horizontal functor comes from the recollement (3.4). Combining recollements (3.6), (4.2) and (4.3) we get the first recollement diagram. While the second recollement diagram could be obtained by recollements (3.8), (4.5) and (4.6). To glue the above two recollement diagrams together, we get the last 2-recollement diagram. This completes the proof.
%\end{proof}

Recall that an $A$-$B$-bimodule $_AM_B$ is called a \emph{Frobenius bimodule} if it is projective as a left $A$- and right $B$-module, and there is an $A$-$B$-bimodule isomorphism
$${{_{B}\Hom_{B^{op}}}(M,B)_{A}}\simeq {_{B}\Hom_{A}}(M,{A})_{A}.$$
Meanwhile, an extension $A\subseteq B$ of algebras is called a \emph{Frobenius extension} if $B$ is projective
as an $A$-module and $B\cong\Hom_{A}(_{A}B,A)$ as a $B$-$A$-bimodule. In this case, both $_{A}B_{B}$ and $_{B}B_{A}$ are Frobenius bimodules. We refer to Kadison \cite{Kad} for more details on this matter.

\begin{prop} (1) Let $T=\left(
                                                                                 \begin{array}{cc}
                                                                                   A & M \\
                                                                                   0 & B \\
                                                                                 \end{array}
                                                                               \right)$ be a triangular matrix algebra with $_AM_B$ a Frobenius bimodule.
                                                                              Then we have the following commutative diagram of 2-recollements
$$\xymatrix@R=4.5ex{0\ar[d]&0\ar[d]&0\ar[d]\\
 \underline{\Gproj A}\ar[d]\ar@/^0.35pc/[r]\ar@/_1.05pc/[r]& \underline{\Gproj T}\ar[d]\ar@/^0.35pc/[l]\ar@/_1.05pc/[l]\ar@/^0.35pc/[r]\ar@/_1.05pc/[r] & \underline{\Gproj B}\ar[d]\ar@/^0.35pc/[l]\ar@/_1.05pc/[l]\\
\D_{sg}(A)\ar[d]\ar@/^0.35pc/[r]\ar@/_1.05pc/[r]& \D_{sg}(T)\ar[d]\ar@/^0.35pc/[l]\ar@/_1.05pc/[l]\ar@/^0.35pc/[r]\ar@/_1.05pc/[r] & \D_{sg}(B)\ar[d]\ar@/^0.35pc/[l]\ar@/_1.05pc/[l]\\
\D_{def}(A)\ar[d]\ar@/^0.35pc/[r]\ar@/_1.05pc/[r]& \D_{def}(T)\ar[d]\ar@/^0.35pc/[l]\ar@/_1.05pc/[l]\ar@/^0.35pc/[r]\ar@/_1.05pc/[r] & \D_{def}(B)\ar[d]\ar@/^0.35pc/[l]\ar@/_1.05pc/[l]\\
0&0&0.}$$

(2) Let $A\subseteq B$ be a Frobenius extension of algebras.
 Assume $T'=\left(
                                                                                 \begin{array}{cc}
                                                                                   A & B \\
                                                                                   0 & B \\
                                                                                 \end{array}
                                                                               \right)$, then we have the following commutative diagram of 2-recollements

$$\xymatrix@R=4.5ex{0\ar[d]&0\ar[d]&0\ar[d]\\
 \underline{\Gproj A}\ar[d]\ar@/^0.35pc/[r]\ar@/_1.05pc/[r]& \underline{\Gproj T'}\ar[d]\ar@/^0.35pc/[l]\ar@/_1.05pc/[l]\ar@/^0.35pc/[r]\ar@/_1.05pc/[r] & \underline{\Gproj B}\ar[d]\ar@/^0.35pc/[l]\ar@/_1.05pc/[l]\\
\D_{sg}(A)\ar[d]\ar@/^0.35pc/[r]\ar@/_1.05pc/[r]& \D_{sg}(T')\ar[d]\ar@/^0.35pc/[l]\ar@/_1.05pc/[l]\ar@/^0.35pc/[r]\ar@/_1.05pc/[r] & \D_{sg}(B)\ar[d]\ar@/^0.35pc/[l]\ar@/_1.05pc/[l]\\
\D_{def}(A)\ar[d]\ar@/^0.35pc/[r]\ar@/_1.05pc/[r]& \D_{def}(T')\ar[d]\ar@/^0.35pc/[l]\ar@/_1.05pc/[l]\ar@/^0.35pc/[r]\ar@/_1.05pc/[r] & \D_{def}(B)\ar[d]\ar@/^0.35pc/[l]\ar@/_1.05pc/[l]\\
0&0&0.}$$

(3)Let $A\subseteq B$ be a Frobenius extension of algebras.
 Assume $T''=\left(
                                                                                 \begin{array}{cc}
                                                                                   B & B \\
                                                                                   0 & A \\
                                                                                 \end{array}
                                                                               \right)$, then we have the following commutative diagram of 2-recollements
$$\xymatrix@R=4.5ex{0\ar[d]&0\ar[d]&0\ar[d]\\
 \underline{\Gproj B}\ar[d]\ar@/^0.35pc/[r]\ar@/_1.05pc/[r]& \underline{\Gproj T''}\ar[d]\ar@/^0.35pc/[l]\ar@/_1.05pc/[l]\ar@/^0.35pc/[r]\ar@/_1.05pc/[r] & \underline{\Gproj A}\ar[d]\ar@/^0.35pc/[l]\ar@/_1.05pc/[l]\\
\D_{sg}(B)\ar[d]\ar@/^0.35pc/[r]\ar@/_1.05pc/[r]& \D_{sg}(T'')\ar[d]\ar@/^0.35pc/[l]\ar@/_1.05pc/[l]\ar@/^0.35pc/[r]\ar@/_1.05pc/[r] & \D_{sg}(A)\ar[d]\ar@/^0.35pc/[l]\ar@/_1.05pc/[l]\\
\D_{def}(B)\ar[d]\ar@/^0.35pc/[r]\ar@/_1.05pc/[r]& \D_{def}(T'')\ar[d]\ar@/^0.35pc/[l]\ar@/_1.05pc/[l]\ar@/^0.35pc/[r]\ar@/_1.05pc/[r] & \D_{def}(A)\ar[d]\ar@/^0.35pc/[l]\ar@/_1.05pc/[l]\\
0&0&0.}$$

\end{prop}
\begin{proof} Since $A\subseteq B$ is a Frobenius extension, it follows that both $_{A}B_{B}$ and $_{B}B_{A}$ are Frobenius bimodules. We only prove (1), because (2) and (3) are its consequences.
As $M$ is Frobenius, it is easy to see $\pd_B\Hom_A(M,A)<\infty$. Following \cite[Theorem 3.4]{HLGZ}, we know that $M\otimes_B-$ and $\Hom_A(M,-)$ preserve Gorenstein projective modules. Hence $M$ is Gorenstein singular and then we get the desired 2-recollement diagram by Corollary \ref{cor:4.9}.
\end{proof}

\bigskip {\bf Acknowledgements}
\bigskip

This research was partially supported by NSFC (Grant No. 11626179, 11671069, 11701455, 11771212), Qing Lan
Project of Jiangsu Province, Jiangsu Government Scholarship for Overseas Studies (JS-2019-328), Shaanxi Province Basic Research Program of Natural Science (Grant No. 2017JQ1012, 2020JM-178) and Fundamental Research Funds for the Central Universities (Grant No. JB160703).

\end{document}